\documentclass[a4paper,leqno,11pt]{article}

\usepackage{hyperref,graphics}
\usepackage{amsmath,bm,stmaryrd,mathtools, xy, subfigure}
\usepackage{array,amsthm}
\usepackage{amscd}
\usepackage{amssymb}
\usepackage{enumerate}
\usepackage{indentfirst}
\usepackage{latexsym}
\usepackage{multicol}
\usepackage{color}

\makeatletter\@addtoreset{equation}{section}\makeatother

\newtheorem{theorem}{Theorem}
\newtheorem{remark}{Remark}[section]
\begingroup
\newtheorem{lemma}{Lemma}[section]
\newtheorem{proposition}[lemma]{Proposition}

\endgroup

\newtheorem{definition}{Definition}

\newcommand{\dx}{{\rm d} {x}}
\newcommand{\dt}{{\rm d} t }

\newcommand{\N}{\mbox{\F N}}

\newcommand{\ds}{\displaystyle}
\newcommand{\beq}[1]{\begin{equation} \label{#1}\ds}
\newcommand{\eeq}{\end{equation}}
\newcommand{\bml}[1]{\beq{#1} \begin{array}{c}\ds}
\newcommand{\eml}{\end{array}\eeq}
\newcommand{\beqq}{\begin{equation*}\ds}
\newcommand{\eeqq}{\end{equation*}}
\newcommand{\bmll}{\beqq \begin{array}{c}\ds}
\newcommand{\emll}{\end{array}\eeqq}

\newcommand{\abs}[1]{\ensuremath{\left| #1 \right|}}
\newcommand{\norm}[1]{\ensuremath{\left\| #1 \right\|}}

\def \d{\mathrm{d}}

\newcommand{\R}{\mathbb{R}}

\def \N{\mathbb{N}}

\begin{document}

\author{Luigi C. Berselli and Elisabetta Chiodaroli}

\title{\bf On the energy equality for the 3D Navier--Stokes equations}

\date{}

\maketitle

\centerline{Dipartimento di Matematica, Universit\`a di Pisa}

\centerline{Via F. Buonarroti 1/c, 56127 Pisa, Italy}

\centerline{email: luigi.carlo.berselli@unipi.it, elisabetta.chiodaroli@unipi.it}

\begin{abstract}
  In this paper we study the problem of energy conservation for the solutions of the
  initial boundary value problem associated to the 3D Navier-Stokes equations, with
  Dirichlet boundary conditions. 

  First, we consider Leray-Hopf weak solutions and we prove some new criteria, involving
  the gradient of the velocity. Next, we compare them with the existing literature in
  scaling invariant spaces and with the Onsager conjecture.
  
  Then, we consider the problem of energy conservation for very-weak solutions,
  proving energy equality for distributional solutions belonging to the so-called Shinbrot
  class. A possible explanation of the role of this classical class of solutions, which is
  not scaling invariant, is also given.

\medskip


\medskip

\end{abstract}
\section{Introduction}
We consider the Navier-Stokes equations in a bounded domain $\Omega\subset\R^3$ with
smooth boundary $\partial \Omega$, and vanishing Dirichlet boundary conditions
\begin{equation} 
  \label{eq:NS}
  \left\{\begin{aligned}
      \partial_t u - \Delta u +(u \cdot \nabla)\, u + \nabla p &=0 \qquad \mbox{in
      }\Omega\times (0,T),
      \\
      \nabla \cdot u&=0 \qquad \mbox{in }\Omega\times (0,T),
      \\
      u&=0 \qquad \mbox{on }\partial\Omega\times (0,T) ,
      \\
      u(\cdot, 0)&=u_0 \qquad \mbox{in }\Omega.
    \end{aligned}\right.
\end{equation}
We study the problem in a bounded domain, but the results are valid also in more general
domains, with minor changes in the proofs.  For simplicity we treat the problem with unit
viscosity and with vanishing external force, but both assumptions are unessential.

We recall that starting with the works of Leray~\cite{Ler1934} and Hopf~\cite{Hop1951} the
notion of \textit{kinetic energy} became of fundamental relevance in the analysis of the
Navier-Stokes equations, since its boundedness represents the very basic \textit{a priori}
estimate, which allows to start the machinery which is used to construct the so-called
Leray-Hopf weak solutions. It is well-known that solutions in this class exist for all
positive times, but their regularity and even their uniqueness are still unproved at
present.  The balance of the kinetic energy formally follows by multiplying the equations by the
solutions itself and by performing some integration by parts, once the needed computations are allowed. On the other hand, the work of Leray
and Hopf showed that the energy inequality below (opposed to the
equality~\eqref{eq:energy-equality}) 
\begin{equation*}
  \frac{1}{2}\int_\Omega|u(t,x)|^2\dx+\int_0^t\int_\Omega|\nabla u(s,x)|^2\,\dx{\rm d}s  
  \leq\frac{1}{2}\int_\Omega|u_0(x)|^2\d x\qquad  \forall\,t\geq0, 
\end{equation*}
can be inferred for weak solutions in the class where globally in time
existence holds. In the latter inequality the sign ``less or equal'' is due to the lack of
regularity of weak solutions and it comes from a limiting process on smoother functions based on weak convergence. We note that under some restrictions on the domain also the
\textit{strong energy inequality} is valid, but this is inessential for the present
research, see e.g.~\cite{Gal2000a}.

The first attempts to determine sufficient conditions implying the validity of the energy
equality~\eqref{eq:energy-equality} in the class of weak solutions --especially in
connection with the problem of uniqueness-- came with a series of papers by
Lions~\cite{Lio1960} and Prodi~\cite{Pro1959}, where the criterion $u\in
L^4(0,T;L^4(\Omega))$ has been identified. The connections with the 2D case and the
results of Kiselev and Lady\v{z}henskaya~\cite{KL1957} highlighted the similarities and
differences between the two and three-dimensional case. A few years later, in the wake of
the celebrated survey of Serrin~\cite{Ser1963}, Shinbrot~\cite{Shi1974} derived the
criterion
\begin{equation}
  \label{eq:Shinbrot}
  u\in L^r(0,T;L^s(\Omega))\qquad \text{with}\quad
  \frac{2}{r}+\frac{2}{s}=1\qquad\text{for}\quad s\geq4,  
\end{equation}
which contains the condition $L^4(0,T;L^4(\Omega))$ as a sub-case, and which extends to a
wide range of exponents the condition for energy conservation.
\begin{remark}
  The condition derived by Shinbrot reads more precisely as
  $\frac{2}{r}+\frac{2}{s}\leq1$. The most interesting is the limiting case and here we
  refer to condition~\eqref{eq:Shinbrot}, since the same results are straightforward if
  the sum is strictly smaller than one.
\end{remark}
It is extremely relevant to observe that the condition~\eqref{eq:Shinbrot} does not depend
on the space dimension $n$ (in particular it is valid also for $n>3$) and moreover --very
remarkably-- it is not scaling invariant, we will elaborate on this in the sequel. We
observe that the results of Lady\v{z}henskaya, Prodi, and Serrin (popularized especially
after the appearance of~\cite{Ser1963}) showed regularity of weak solutions (and hence
validity of the energy equality) if the scaling invariant condition
\begin{equation}
  \label{eq:Serrin}
  u\in L^r(0,T;L^s(\Omega))\qquad \text{with}\quad \frac{2}{r}+\frac{3}{s}=1\qquad
  \text{for}\quad s>3, 
\end{equation}
holds true (see also Sohr~\cite{Soh1984}). A nice survey of these classical results can be
found in Galdi~\cite{Gal2000a}.

The fact that Shinbrot condition does not fit in with the framework of scaling invariance
(and also that~\eqref{eq:Shinbrot} reduces to~\eqref{eq:Serrin} as $s\to+\infty$) makes
the understanding and the explanation of~\eqref{eq:Shinbrot} more involved. We will give
later on in Section~\ref{sec:comparison-very-weak} a possible interpretation, which is
based on a different scaling and which highlights the importance of the convective term as
responsible of the non-linear character of the equations, which seems in this case more
relevant than the non-local nature induced by the divergence-free constraint and by the
pressure. Indeed, it seems that the ``hyperbolic'' part of the equations, plays a
fundamental role in the redistribution of the energy, not at the Fourier level, but in a
sort of regularity exchange, see Section~\ref{sec:comparison-very-weak}.

The essential difference between the energy equality and the energy inequality is the
presence of anomalous dissipation due to the presence of non-linearity. The dissipation
phenomenon is expected to be connected with the possible roughness of the solutions.  The
importance of energy conservation/dissipation (especially in the limit
of vanishing viscosity) came from Onsager's work~\cite{Ons1949}. Moreover, for the
Navier-Stokes equations the possible connection between energy conservation and uniqueness
of weak solutions still represent an interesting open problem. Furthermore, we observe that
the energy equality, will imply in the case of Galerkin approximation to the solution the 
convergence of the norms $\|u^m\|_{L^2(0,T;H^1(\Omega)}$, which combined with the weak
convergence will produce the strong convergence. This will have also relevant consequences
when these methods are used in the numerical approximation of the solutions.

The famous Onsager conjecture for the Euler and Navier-Stokes equations predicts the
threshold regularity for energy conservation.  In this direction, energy conservation for
$C^{0,1/3}$-solutions of the Euler equations has been recently addressed in the work of
Isett~\cite{Ise2018} and also Buckmaster, De Lellis, Sz\'ekelyhidi, and
Vicol~\cite{BDLSV}.  On the other hand the role of Besov-H\"older continuous spaces in the
energy conservation of the viscous Navier-Stokes equations was studied by Cheskidov,
Constantin, Friedlander and Shvydkoy~\cite{CCFS2008}, who proved that weak solutions in
the following class
\begin{equation*}
  u\in L^3(0,T;B^{1/3}_{3,\infty}(\R^{3})),
\end{equation*}
conserve the energy. (Cf. also Constantin, E, and Titi~\cite{CET1994} for the Euler
equations) The latter result has been recently extended by Cheskidov and Luo~\cite{CL2018}
to the class
\begin{equation*}
  u\in L^\beta_w(0,T;B^{\frac{2}{\beta}+\frac{2}{p}-1}_{p,\infty}(\R^{3}))
  \qquad\text{with}\quad \frac{2}{p}+\frac{1}{\beta}<1  \qquad\text{for}\quad
  1\leq\beta<p\leq\infty.
\end{equation*}
where $L^{\beta}_{w}$ denotes the weak (Marcinkiewicz) space, while $B^{s}_{p,q}(\R^3)$ are the
standard Besov spaces.

A similar approach, based on Fourier methods, allowed Cheskidov, Friedlander, and
Shvydkoy~\cite{CFS2010} to prove the following sufficient condition for energy
conservation (here $A$ denotes the Stokes operator associated to the Dirichlet boundary
conditions)
\begin{equation*}
  A^{5/12}u\in L^{3}(0,T;L^{2}(\Omega)),
\end{equation*}
which turns out to be --in terms of scaling-- less strict than Shinbrot
one~\eqref{eq:Shinbrot}. In fact, the latter criterion by Cheskidov, Friedlander, and
Shvydkoy is \textit{equivalent in terms of scaling} to $u\in
L^3(0,T;L^{9/2}(\Omega))$. The comparison and discussion of the various results will be
given in detail in Section~\ref{sec:comparison-weak}.  We also recall that
Farwig~\cite{Far2014} proved the following sufficient condition for energy conservation
\begin{equation*}
  A^{1/4}u\in L^{3}(0,T;L^{18/7}(\Omega)),
\end{equation*}
which turns out to have the same scaling.

Our first aim is to improve the above results by looking for regularity requirements
sufficient to prove energy conservation and weaker than the existing ones. The
improvements coming from our results can be measured in terms of scaling, see
Theorem~\ref{thm:grad-ranges}. Moreover, we will use tools of functional analysis
accessible to an extremely wide range of researchers, since we mainly focus on showing
conditions involving usual derivatives in Lebesgue spaces, without entering the technical
difficulties of more complex functional settings. This complies with the spirit of looking
for refinements of basic results in classical functional spaces. Possible further
developments in slightly sharper, but more complicated spaces (but at the same level of
scaling) are not excluded, but beyond our goals.

In order to present our first main result we recall, for the reader not acquainted with
all the technicalities necessary when dealing with genuine Leray-Hopf weak solutions of
the Navier-Stokes equations, that one main step in the analysis of the energy equality is
to rigorously prove the following equality
\begin{equation}
  \label{eq:fundamental} 
  \int_0^T\int_\Omega (u\cdot \nabla)\, u \cdot u\,\dx\d t =0.
\end{equation}
The latter equality is \textit{formally} valid if $u$ is a divergence-free vector field
tangential to the boundary. The calculations are justified once $u$ is smooth enough to
ensure that the above space-time integral is finite, but this is not the case if $u$ is a
weak solution in the 3D case.

The novelty of our approach is to look for conditions involving the gradient of the
velocity, instead of the velocity itself, which allow to conclude that the above integral is
finite.

The first result of this paper is the following criterion. 
\begin{theorem}
  \label{thm:grad-ranges}
  Let $u_0\in H$ and let $u$ be a Leray-Hopf weak solution of~\eqref{eq:NS}
  corresponding to $u_0$ as initial datum.  Let us assume that one the following conditions
  is satisfied
\begin{itemize}
  \item[(i)]  $\nabla u\in
  L^\frac{q}{2q-3}(0,T;L^q(\Omega))$,\qquad for\quad  $\frac{3}{2}< q<\frac{9}{5}$;
  \item[(ii)] $\nabla u\in L^\frac{5q}{5q-6}(0,T;L^q(\Omega))$,\qquad for \quad$\frac{9}{5}\leq q
    \leq 3$;
  \item[(iii)] $\nabla u\in L^{1+\frac{2}{q}}(0,T;L^q(\Omega))$, \qquad for\quad  $q> 3$.
  \end{itemize}
  Then, the velocity $u$ satisfies the energy equality
\begin{equation}
  \label{eq:energy-equality}
  \frac{1}{2}\int_\Omega|u(t,x)|^2\dx+\int_0^t\int_\Omega|\nabla u(s,x)|^2\,\dx{\rm d}s 
  =\frac{1}{2}\int_\Omega|u_0(x)|^2\d x\qquad\forall\,t\geq0.
\end{equation}
\end{theorem}
These results will be proved in Section~\ref{sec:weak}, but the comparison with the
previous ones known in literature will be be discussed already in
Section~\ref{sec:comparison-weak}.  The overall strategy is, to assume that $\nabla u\in
L^q(\Omega)$ for some fixed $q$, and to employ all the available regularity coming from
the definition of weak solutions, so to find the smaller exponent in the time variable
which allows to show~\eqref{eq:fundamental}.
\begin{remark}
  A similar approach had already been implemented also by Farwig and
  Taniuchi~\cite{FT2010}, focusing on the degree $k$ of (fractional) smoothness of $u$ and
  on the correct definition of spaces in the case of general unbounded domains. In
  particular, their result for $k=1$ has the same scaling as ours in the case $q=9/5$.
  Very recently, our approach of working with levels of regularity for the gradient of the
  velocity has inspired the paper of {Beir{\~a}o da Veiga} and Yang~\cite{BY2019} where
  they treat Newtonian and non-Newtonian fluids.
\end{remark}
\begin{remark}
  Even if not explicitly stated, it is clear that any of the conditions on the gradient in
  Theorem~\ref{thm:grad-ranges} implies, directly by the standard theory of traces, an
  extra condition on the initial datum. To keep the paper self-contained and more
  understandable we do not elaborate on this technical point, which can be probably
  improved by using the theory of weighted estimates as in Farwig, Giga, and
  Hsu~\cite{FGH2016} or by dealing with conditions valid on any time interval of the type
  $[\varepsilon,T]$ (with $\varepsilon>0$), as done in the recent paper by
  Maremonti~\cite{Mar2018}.
\end{remark}

\bigskip

In the last part of the paper (cf.~Section~\ref{sec:very-weak}) we will consider the
problem of energy conservation for solutions which are very-weak, see
Section~\ref{sec:classes-solutions} for the precise definition. In this case there is not
any available regularity on $u$, apart the solution being in
$L^2_{loc}((0,T)\times\Omega)$. The interest for very-weak solutions dates back to
Foias~\cite{Foi1961}, who proved their uniqueness under condition~\eqref{eq:Serrin}. Later
Fabes, Jones, and Rivi{\`e}re~\cite{FJR1972} proved the existence of very-weak solutions
for the Cauchy problem, while the case of the initial-boundary value problem has been
studied mainly starting from the work of Amann~\cite{Ama2000}. As usual when dealing with
very-weak solutions a duality argument can be employed to show uniqueness, by using
properties of the adjoint problem (which in case of the Navier-Stokes equations, is a
backward Oseen type-problem). The connection with the energy equality has been very
recently developed by Galdi~\cite{Gal2018}, who showed that very-weak solution in the
Lions-Prodi class $L^4(0,T;L^4(\Omega))$, and with initial data in $L^2(\Omega)$,
satisfy~\eqref{eq:energy-equality}. It is relevant to observe that the duality argument is
used to improve the known regularity of the solution, in order to use the previously
established results for usual weak solutions. This has been used also in a different
context in~\cite{BG2004}, but the approach we follow here takes also inspiration from a
bootstrap argument as used in Lions and Masmoudi~\cite{LM2001} for results concerning
\textit{mild solutions}.

The second result we prove is then the following.
\begin{theorem}
  \label{thm:very-weak}
  Let $u\in L^2_{loc}(\Omega \times (0,T))$ be a distributional solution
  of~\eqref{eq:NS}. If the initial datum $u_0\in H$ and if~\eqref{eq:Shinbrot} holds true, 
then the energy equality~\eqref{eq:energy-equality} is satisfied.
\end{theorem}
The interpretation of the latter result and its connection with the scaling of space and
time variables will be given in Section~\ref{sec:comparison-very-weak}.  The techniques
employed here are inspired by the approach followed in Galdi~\cite{Gal2018,Gal2018b}.

Summarizing the results, the sufficient conditions required for the energy equality are
stronger if the requested regularity of the solution $u$ is weaker. In particular, the
less strict requirement of the solution (being just a distributional solution) seems to
require more restrictive conditions then those in Theorem~\ref{thm:grad-ranges}, which are
nevertheless the same as those classically found in~\cite{Pro1959,Lio1960,Shi1974} for
Leray-Hopf weak solutions.

\bigskip

\noindent\textbf{Acknowledgments} 
The authors warmly thank Prof. Herbert Amann for having provided in a private
communication, and prior to publishing, the results of Lemma~\ref{lem:embedding}, which
will appear soon in Ref.~\cite{Ama2018}.  

The research that led to the present paper was partially supported by a grant of the group
GNAMPA of INdAM and by the project of the University of Pisa within the grant
PRA$\_{}2018\_{}52$~UNIPI \textit{Energy and regularity: New techniques for classical PDE
  problems.}
\section{Functional setting and comparison with previous results}
\label{sec:preliminaries}
In this section we first introduce the notation and the precise definitions of solutions
we want to deal with, then we compare our results with the ones in the existing
literature.  We will use the customary Sobolev spaces
$(W^{s,p}(\Omega),\,\|\,.\,\|_{W^{s,p}})$ and we denote the $L^p$-norm simply by
$\|\,.\,\|_p$ and since the Hilbert case plays a special role we denote the
$L^2(\Omega)$-norm simply by $\|\,.\,\|$. For $X$ a Banach space we will also denote the
usual Bochner spaces of functions defined on $[0,T]$ and with values in $X$ by
$(L^p(0,T;X),\|\,.\,\|_{L^p(X)})$. In the case $X=L^q(\Omega)$ we denote the norm of
$L^{p}(0,T;L^{q}(\Omega))$ simply by $\|\,.\,\|_{p,q}$.
\subsection{On Leray-Hopf weak and very-weak solutions}
\label{sec:classes-solutions}
For the variational formulation of the Navier-Stokes equations~\eqref{eq:NS}, we introduce
as usual the space $\mathcal{V}$ of smooth and divergence-free vectors fields, with compact
support in $\Omega$.  We then denote the completion of $\mathcal{V}$ in $L^2(\Omega)$ by
$H$ and the completion in $H^1_0(\Omega)$ by $V$. The Hilbert space $H$ is endowed with
the natural $L^2$-norm $\norm{\,.\,}$ and inner product $(\cdot,\cdot)$, while $V$ with
the norm $\norm{\nabla v}$ and inner product $((u,v)):=(\nabla u, \nabla v)$. As usual, we
do not distinguish between scalar and vector valued functions. The dual pairing between
$V$ and $V'$ is denoted by $\langle\cdot, \cdot\rangle$, and the dual norm by
$\norm{\,.\,}_\ast$.

\begin{definition}[Leray-Hopf weak solutions]
\label{def:weak}
  A vector field $u\in L^\infty(0,T; H)\cap L^2(0,T; V)$ is a Leray-Hopf weak solution to
  the Navier-Stokes equations~\eqref{eq:NS} if
\begin{itemize}
	\item [(i)]  $u$ is a solution of ~\eqref{eq:NS} in the sense of distributions, i.e.
  \begin{equation*}
    \int_0^T(u,\partial_t \phi)-(\nabla u,\nabla \phi)-((u\cdot\nabla)\,
    u,\phi)\,\dt=-(u_0,\phi(0)), 
  \end{equation*}
  for all $\phi\in C^\infty_0([0,T[\times \Omega)$ with $\nabla\cdot \phi=0$;
\item [(ii)] $u$ satisfies the global energy inequality
  \begin{equation}
    \label{eq:energy-inequality}
    \frac{1}{2}\|u(t)\|^2+\int_0^t\norm{\nabla u(s)}^2\,\ds\leq\frac{1}{2}\|u_0\|^2\qquad
    \forall\,t\geq0;
  \end{equation}
\item [(iii)]  the initial datum is attained in the strong sense of $L^2(\Omega)$
  \begin{equation*}
    \|u(t)-u_0\|\to0\qquad t\to0^+.
  \end{equation*}
\end{itemize}
\end{definition}
Since the works of Leray and Hopf~\cite{Ler1934,Hop1951} for the Cauchy problem and for
the initial boundary vale problem, it is well-known that for initial data in $H$, and for
all $T>0$ there exists at least a weak solution in the above sense. An outstanding open
problem is that of proving (or disproving) the uniqueness and regularity of weak
solutions, see Constantin and Foias~\cite{CF1988} and Galdi~\cite{Gal2000a}.

\medskip 

In the last section of the paper we deal with another less-restrictive notion of solution,
which allows for infinite energy.
\begin{definition}[very-weak solutions]
\label{def:very-weak}
 A vector field $u\in L^2_{loc}((0,T)\times \Omega)$ is a very-weak solution to the
 Navier-Stokes equations  
 if 
\begin{itemize}
\item [(i)]  the following identity 
 \begin{equation*}
   \int_0^T    (u,\partial_t \phi)+(u,\Delta\phi)+(u,(u\cdot\nabla)\, \phi)\,\dt=-(u_0,\phi(0)),
 \end{equation*}
 holds true for all $\phi\in \mathcal{D}_T$,where
 \begin{equation*}
   \mathcal{D}_T:=\left\{
     \begin{aligned}
       &\phi\in C^\infty([0,T]\times\overline{\Omega}), \text{with support contained in a
         compact set of }[0,T]\times\overline{\Omega}, 
       \\
       &\text{such that } \nabla\cdot \phi=0\text{ in }\Omega ,\ \phi=0\text{ on
       }\partial\Omega\text{ and }\phi(T) = 0 \text{ in }\Omega
     \end{aligned}
\right\};
 \end{equation*}
\item[(ii)]$\nabla\cdot u=0$ in the sense of $\mathcal{D}'(\Omega)$, that in the sense of
  distributions, for a.e. $t\in [0,T]$.
\end{itemize}
\end{definition}
As mentioned before, the existence of very-weak solutions in the case of the whole space
has been treated in~\cite{FJR1972}, but the notion in the 3D time-evolution case in a
bounded domain (especially with analysis of the meaning of boundary conditions) is
analyzed for instance in Amann~\cite{Ama2000} and Farwig, Kozono, and
Sohr~\cite{FKS2007a}, with emphasis on the non-homogeneous problem, too.
%
%
\subsection{Comparison with previous results}
In this section we compare our results with the ones already present in literature. From
the physical point of view the most relevant results are those related with large $q$,
especially with $q>3$, being $3$ the space dimension.
\subsubsection{Energy conservation and Onsager conjecture}
\label{sec:comparison-weak}
As mentioned in the introduction, the validity of~\eqref{eq:energy-equality} has a strong
connection with Onsager conjecture about the threshold regularity of $C^{1/3}$-H\"older
continuity of the velocity which allows for energy conservation.  While for the 3D Euler
equations the conjecture is basically solved in both directions, for the 3D Navier-Stokes
equations much has still to be done.
This is a strong motivation to investigate further the energy equality for the 3D
Navier-Stokes equations and to understand the anomalous energy dissipation phenomenon in
this case.  Our result involves Sobolev regularity of the velocity, but this can be
compared to H\"older regularity as we will explain in the following. In the range $q>3$,
standard Sobolev embedding theorems imply $W^{1,q}(\Omega)\subset
C^{0,1-3/q}(\overline{\Omega})$, and $W^{1,q}(\Omega)$ is a \textit{proper} subset of
$C^{0,1-3/q}(\overline{\Omega})$. On the other hand, in order to compare our result with
previous ones, we can consider the two spaces as if they would be equivalent in a certain
sense.  Indeed, Sobolev norms measure a sort of combination of three properties of a
function: height (amplitude), width (measure of the support), and frequency (inverse
wavelength). Roughly speaking, if a function has amplitude $A$, is supported on a set of
volume $V$, and has oscillations with frequency $N$, then the $W^{k,p}$-(homogeneous) norm
is of the order of $A N^k V^{1/p}$. A key quantity is the ``weight'' $k-\frac{n}{p}$ of
the Sobolev space $W^{k,p}(\R^n)$: If one scales the metric by a constant $R>0$ then this
transformations scales the $W^{k,p}$-norm by $R^{n/p-k}$ and the $C^{m,\alpha}$-norm by
$R^{-(m+\alpha)}$, showing the relevance of the weights.  Hence, from the point of view of
Sobolev embedding and fractional regularity it is sound to consider as very close (and so
almost equivalent) in the 3D case the spaces $W^{s_1,p_1}$ and $W^{s_2,p_2}$, when
$1/p_1-s_1/3=1/p_2-s_2/3$, and they both embed in $C^{0,\alpha}$, for
$\alpha:=s_{i}-3/p_{i}$, whenever this quantity is positive.  Also in terms of
interpolation theory, these spaces behave in the same way, as can be seen with the DeVore
diagrams~\cite{Dev1998}. Hence, we can also consider $C^{0,1-3/q}(\overline{\Omega})$ as a
rough --but meaningful-- measure of the classical regularity of $W^{1,q}(\Omega)$.
 
By embedding, results from Theorem~\ref{thm:grad-ranges} in the range $q>3$ are very close
to the condition
\begin{equation*}
u\in L^{1+\frac{2}{q}}(0,T;C^{0,1-\frac{3}{q}}(\overline{\Omega})).
\end{equation*}
In particular, by taking $q=9/2$ we obtain, as class of solutions conserving the energy,
that with scaling comparable to
\begin{equation*}
  u\in L^{\frac{13}{9}}(0,T;C^{0,1/3}(\overline{\Omega})),
\end{equation*}
which improves the previously cited results
\begin{equation*}
  \begin{aligned}
    &  u\in L^3(0,T;B^{1/3}_{3,\infty}(\R^{3})) \qquad \,\text{  from Ref.~\cite{CCFS2008}};
    \\
    &  u\in L^\frac{3}{2}_w(0,T;B^{1/3}_{\infty,\infty}(\R^{3}))\qquad \text{from
    Ref.~\cite{CL2018},  when $p=\infty$, and $\beta=3/2$}. 
  \end{aligned}
\end{equation*}
In the space variables the two functional spaces are very close to
$C^{0,1/3}(\overline{\Omega})$, but $3>3/2=1.5>13/9=1.\overline{4}$.

We also warn again the reader that our results are obtained by embedding, hence they are
valid for a proper subset of $C^{0,1/3}(\overline{\Omega})$ and there is not any direct
connection between H\"older and Sobolev regularity (recall that the Weierstrass function
$\sum_{n\in\N} a^n \cos(b^n \pi x)$, $0<a<1$, is $-\log(a)/\log(b)-$H\"older continuous
but is not even of bounded variation). Nevertheless, in terms of scaling our results
present a better behavior as compared with the previous ones present in literature. This
suggests that there could be also room for further improvements.
\subsubsection{The case $q<3$}
\label{sec:comparison-very-weak00}
In order to explain the comparison (we will make later on) we also wish to recall the notion
of scaling invariance for space-time functions. By interpolation one can show that
Leray-Hopf weak solution have the following regularity
\begin{equation*}
  u\in L^r(0,T;L^s(\Omega))\qquad\text{with}\quad \frac{2}{r}+\frac{3}{s}=\frac{3}{2},\qquad
  \text{for}\quad  2\leq s\leq6. 
\end{equation*}
Several results (starting again from the classical work of Lady\v{z}henskaya, Prodi, and
Serrin) concern uniqueness and regularity with scaling invariant conditions on solutions.
In particular, if a weak solution satisfies condition~\eqref{eq:Serrin}
(see~\cite{Ser1963} for example), then it becomes unique, strong (namely the gradient
belongs to $L^{\infty}(0,T;L^{2}(\Omega))$), smooth, and it satisfies the energy equality.

Full regularity of weak solutions follows also under alternative assumptions on the
gradient of the velocity $\nabla u$. More specifically, if
\begin{equation}
  \label{eq:gradient-scaling}
  \nabla  u\in L^p(0,T;L^q(\Omega))\qquad\text{with}\quad
  \frac{2}{p}+\frac{3}{q}=2,\qquad\text{for}\quad  q>\frac{3}{2},  
\end{equation}
then weak solutions are regular, see Beir\~ao da Veiga~\cite{Bei1995a} and
also~\cite{Ber2002a} for the problem in a bounded domain.  In $\R^{3}$ standard Sobolev
embeddings imply that if $\nabla u \in L^p(0,T; L^q(\Omega))$ then $u\in L^p(0,T;
L^{q^*}(\Omega))$ where $\frac{1}{q^\ast}=\frac{1}{q}-\frac{1}{3}$.
We recall that for weak solutions $\nabla u$ is simply $(x,t)$-square-integrable and
$\frac{2}{2}+\frac{3}{2}=\frac{5}{2}>2$.

The class defined by~\eqref{eq:Serrin} is important from the
point of view of the relation between scaling invariance and partial regularity of weak
solutions. In fact, if a pair $(u,p)$ solves~\eqref{eq:NS}, then so does the family
$\{(u_\lambda,p_\lambda)\}_{\lambda>0}$ defined by
\begin{equation}
  \label{eq:parabolic-scaling}
  u_\lambda(t,x):=\lambda u(\lambda^2 t,\lambda x)\qquad \text{and}
  \qquad p_\lambda(x,t):=\lambda^2 p(\lambda^2 t,\lambda  x). 
\end{equation}
Scaling invariance means that
$\|u_\lambda\|_{L^r(0,T/\lambda^2;L^s(\Omega_\lambda))}=\|u\|_{L^r(0,T;L^s(\Omega))}$ and
this happens if and only if $(r,s)$ satisfy~\eqref{eq:Serrin}. (Here
$\Omega_\lambda:=\{x/\lambda:\ x\in \Omega\}$.) Likewise, the scaling invariance of
$\nabla u$ occurs if and only if~\eqref{eq:gradient-scaling} holds true. We will use these
notions to compare classical and more recent results with the new ones from
Theorem~\ref{thm:grad-ranges}.

If $u\in L^{4}(0,T;L^{4}(\Omega))$ then, \textit{in terms of scaling}, this regularity lies in between
the class of existence and that of regularity since
\begin{equation*}
1=\frac{2}{4}+\frac{2}{4}<  \frac{2}{4}+\frac{3}{4}=\frac{5}{4}<\frac{3}{2}.
\end{equation*}
We note that also the class~\eqref{eq:Shinbrot} identified by Shinbrot shares the same
property, even if in terms of scaling
\begin{equation*}
1<\frac{2}{r}+\frac{3}{s}=\frac{2}{r}+\frac{2}{s}+\frac{1}{s}=1+\frac{1}{s}<\frac{3}{2},
\end{equation*}
hence, as $s$ increases, this class becomes closer and closer to that of regularity given
by~\eqref{eq:Serrin}.  A natural question is that if we can lower down the level of
regularity needed for a weak solution, in order to satisfy the energy equality.

The new results we present are related with the aim of finding sufficient conditions for
the energy equality involving the gradient of the velocity. 

The fact that condition~\eqref{eq:Shinbrot} is not scaling invariant makes it possible to
conjecture that perhaps some threshold can be broken by a more precise inspection of the
calculations. In fact, the aforementioned results in~\cite{CFS2010,Far2014} identified the
sufficient conditions
\begin{equation*}
  A^{5/12}u\in L^{3}(0,T;L^{2}(\Omega))\qquad\text{or}\qquad   A^{1/4}u\in
  L^{3}(0,T;L^{18/7}(\Omega)). 
\end{equation*}
which are --in terms of scaling-- both comparable with
\begin{equation*}
  u\in  L^{3}(0,T;L^{9/2}(\Omega)),
\end{equation*}
which is less restrictive than~\eqref{eq:Shinbrot}, since
\begin{equation*}
  1< \frac{2}{3}+\frac{2}{9/2}=\frac{10}{9},
\end{equation*}
but at the same time the space $L^{3}(0,T;L^{9/2}(\Omega))$ it is still between the class
of existence and that of regularity, being
$1<\frac{2}{3}+\frac{3}{9/2}=\frac{4}{3}<\frac{3}{2}$.
%

 The ranges obtained in Theorem~\ref{thm:grad-ranges} have respectively the following
  properties, which turn out easy to be compared with the
  condition~\eqref{eq:gradient-scaling}.

We have in fact that: 
  \begin{itemize}
  \item[(i)] it holds $2<\frac{2}{p}+\frac{3}{q}=4-\frac{3}{q}<\frac{5}{2}$, \qquad for \quad
    $\frac{3}{2}< q<\frac{9}{5}$;
  \item[(ii)] it holds $2<\frac{2}{p}+\frac{3}{q}=2+\frac{3}{5q}<\frac{5}{2}$, \qquad for \quad
    $\frac{9}{5}\leq q< 3$;
  \item[(iii)] it holds $2<\frac{2}{p}+\frac{3}{q}=\frac{2q}{q+2}+\frac{3}{q}<\frac{5}{2}$,
    \qquad for \quad $3\leq q<6$.
  \end{itemize}
  Thus, we have that for $\frac{3}{2}<q<6$ our conditions imply range of exponents which are
  not those of scaling invariance~\eqref{eq:gradient-scaling}, hence not those implying full
  regularity of the solutions (and consequently also energy equality in a trivial way).
%
\begin{remark}
  In cases (i) and (ii), i.e. when $\frac{3}{2}<q<3$, the standard Sobolev embedding tells
  us that $u\in L^p(0,T; L^{q^\ast}(\Omega))$ where $q^\ast=\frac{3q}{3-q}$ and $p$ is
  given as a function of $q$ by Theorem~\ref{thm:grad-ranges}. This means that by case (i)
  and case (ii) the exponent of integrability $q^*$ in the space variable for the solution $u$
  can range from $3$ to $+\infty$.
\end{remark}
%
  Recalling that $q^\ast=\frac{3q}{3-q}$, the ranges obtained in
  Theorem~\ref{thm:grad-ranges} have respectively the following properties:
\begin{itemize}
\item[(i)] it holds $1<\frac{2}{p}+\frac{2}{q^\ast}=\frac{2(5q-6)}{3q}$, \qquad for\quad 
  $\frac{12}{7}<q<\frac{9}{5}$;
\item[(ii)] it holds $1<\frac{2}{p}+\frac{2}{q^\ast}=\frac{2(10q-3)}{15q}$, \qquad for\quad
  $\frac{9}{5}\leq q<3$;
\item[(iii)] it holds $1<\frac{2}{p}+\frac{2}{q^\ast}=\frac{2(2q-3)}{3q}$, \qquad
  for \quad $q\geq3$.
\end{itemize}
Thus, we showed that our range of exponents improves those in Shinbrot
condition~\eqref{eq:Shinbrot}.
We observe that Shinbrot condition for the space integrability
of $u$ (exponent $\geq4$) corresponds to $q>\frac{12}{7}$ in our classification.
%
\begin{remark}
  The ``best exponent,'' where best is measured in terms of giving the quantity
  $\frac{2}{p}+\frac{2}{q}$ as large as possible, turns out to be $q=9/5$. This implies that
  by embedding $(9/5)^*=9/2$ which gives as sufficient condition
\begin{equation*}
  u\in L^3(0,T;W^{1,9/5}(\Omega))\subset L^3(0,T;L^{9/2}(\Omega)),
\end{equation*}
that is at the same level of scaling of~\cite{CFS2010,Far2014}, even if the
various conditions are not directly comparable each other.
\end{remark}
%
We further remark that, in the range $q>\frac{12}{7}$, our result improves also the ranges
obtained by Leslie and Shvydkoy in~\cite{LS2018}.  Indeed, they prove
(see~\cite[Theorem~1.1]{LS2018}) the validity of the energy equality for $u\in
L^r(0,T;L^s(\Omega))$ with $3\leq r \leq s$ and
\begin{equation*}
  \frac{1}{r}+\frac{1}{s}\leq \frac{1}{2},
\end{equation*}
while, for $q>\frac{12}{7}$ we are assuming $\nabla u\in L^p(0,T; L^q(\Omega))$ and hence
$u\in L^p(0,T; L^{q^\ast}(\Omega))$ with $p\leq q^\ast$ and
\begin{equation*}
  \frac{1}{p}+\frac{1}{q^\ast}>\frac{1}{2}.
\end{equation*}
However, authors in~\cite{LS2018} studied also the case $s<3$ corresponding in our case to
$q<3/2$ which is not covered here.  Moreover, for the case $3\leq s< r$, they can prove
the energy equality with exponents $r$ and $s$ satisfying
$\frac{1}{r}+\frac{1}{s}<\frac{1}{2}$: this matches our interval $3/2\leq q <12/7$, where
$\frac{1}{p}+\frac{1}{q^\ast}< \frac{1}{2},$ thus showing that we are not
improving~\cite{LS2018} in this last range of exponents.

\begin{figure}[htbp]
\begin{center}
\scalebox{0.5}{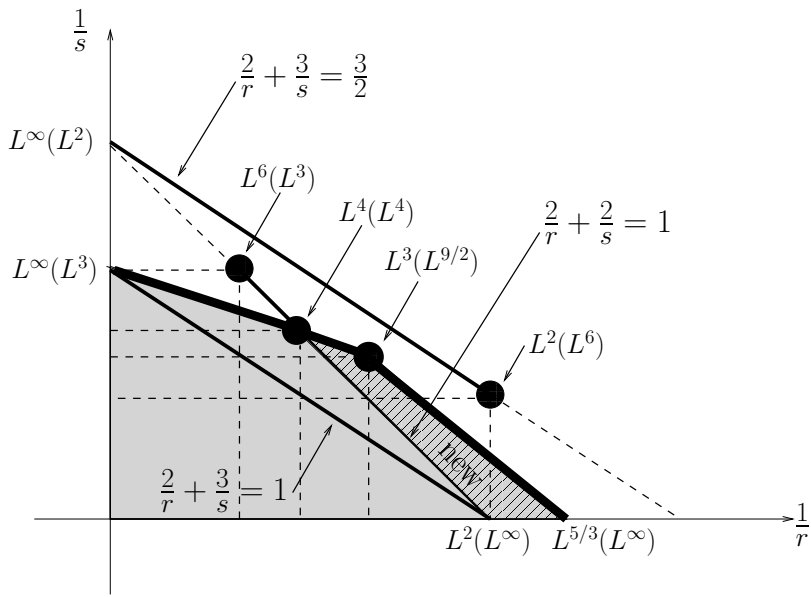 }
\caption{The shaded area corresponds to the ranges of Theorem~\ref{thm:grad-ranges} and in
  particular the dashed one corresponds to new ranges of exponents not obtained before in
  the literature.}
\label{fig:comparison}
\end{center}
\end{figure}

%
\subsection{On the Shinbrot condition~\eqref{eq:Shinbrot} and very-weak
  solutions} \label{sec:comparison-very-weak}
In the case of very-weak solutions, the sufficient condition for energy equality we find
is the same as~\eqref{eq:Shinbrot}, hence showing that it has a universal role, since it
applies to distributional solutions which can be outside of any Lebesgue space, hence not
the common Leray-Hopf weak solutions.

We observe that the condition~\eqref{eq:Shinbrot} breaks all the standard theory of
scaling invariance, even if we consider different families of scaling transformation. In
fact, beside the standard parabolic scaling~\eqref{eq:parabolic-scaling}, it is well-known
that useful different approaches are those concerning invariance of the equation under the
following space-time transformation, indexed by $\alpha\in \R$
\begin{equation}
  \label{eq:general-scaling}
  u_{\lambda,\alpha}(t,x):=\lambda^{\alpha} u(\lambda^{\alpha+1} t,\lambda
  x)\qquad\text{and}\qquad 
  p_{\lambda,\alpha}(t,x):=\lambda^{2\alpha} p(\lambda^{\alpha+1} t,\lambda x). 
\end{equation}
Note that the transformation in~\eqref{eq:parabolic-scaling} corresponds to $\alpha=1$. 
On varying $\alpha$, with these transformations it is possible to extract some heuristic
and useful information from the equations. In particular, they keep for all $\alpha\in \R$ the
material derivative operator $\frac{D}{Dt}:=\partial_t+u\cdot\nabla$ unchanged, while the
viscosity changes from $\nu$ to $\lambda^{\alpha-1}\nu$. This explains the relevance of
these transformation in the study of Euler equations or also in presence of very small
viscosities. We point out that the case $\alpha=-1/3$ determine the so-called
\textit{Kolmogorov scaling}, which has connections with the conservation of energy input
for stochastic statistically stationary solutions, stochastic equations, and Large Eddy
Simulation for turbulence modeling. See~\cite{BF2010,BIL2006}, Flandoli, Gubinelli,
Hairer, and Romito~\cite{FGHR2008}, and Kupiainen~\cite{Kup2003}.  If we try to consider
the scaling of norms under the transformation~\eqref{eq:general-scaling} we obtain that
\begin{equation*}
  \|u_{\lambda,\alpha}\|_{L^{r}(0,T/\lambda^2;L^{s}(\Omega_{\lambda}))}
  =\lambda^{\alpha-\frac{3}{s}-\frac{\alpha+1}{r}}\|u\|_{L^{r}(0,T;L^{s}(\Omega))}
\end{equation*}
hence the norms cannot be invariant for the relevant values of $\alpha$, since
$\alpha-\frac{3}{s}-\frac{\alpha+1}{r}$ never vanishes for both $\alpha=1$ and
$\alpha=-1/3$. A possible explanation or justification of the
condition~\eqref{eq:Shinbrot} will come from a different analysis.

In particular, in condition~\eqref{eq:Shinbrot} the space and time variables have the same
``strength,'' as in hyperbolic equations, while in our studies the viscosity and hence the
parabolic part play a fundamental role. To understand in which situation the hyperbolic
nature of the convection plays a major role we have to consider the maximal regularity
type results for the Stokes problem, which will be used to handle the adjoint
problem. Coming back to Solonnikov~\cite{Sol1976} (see also Giga and Sohr~\cite{GS1991})
it is well-known that for a wide class of domains the following result holds true.
\begin{theorem}
\label{thm:maximal-regularity}
Let $\Omega\subset\R^{3}$ be smooth and bounded (or exterior or the full-space or even
the half-space) and let $\mathcal{F}\in L^{\alpha}(0,T;L^{\beta}(\Omega))$, with
$1<\alpha,\, \beta<\infty$. Then, the boundary initial value problem associated to the
linear Stokes system
  \begin{equation}
    \label{eq:linear-Stokes}
    \begin{aligned}
      \partial_{t}v-\Delta v+\nabla \pi&=\mathcal{F}\qquad \text{in }(0,T)\times\Omega, 
      \\
      \nabla \cdot v&=0 \qquad \text{in }(0,T)\times\Omega,
      \\
      v(t,x)&=0 \qquad \text{on }(0,T)\times\partial\Omega,
      \\
      v(0,x)&=0 \qquad \text{in }\Omega,
    \end{aligned}
  \end{equation}
has a unique solution $(v,\pi)$ such that
\begin{equation}
  \label{eq:max-regularity}
  \exists\, c>0\qquad \|\partial_{t}v\|_{L^{\alpha}(L^{\beta})}+\|P\Delta
  v\|_{L^{\alpha}(L^{\beta})}+\|\nabla \pi\|_{L^{\alpha}(L^{\beta})}\leq
  c\,\|\mathcal{F}\|_{L^{\alpha}(L^{\beta})}, 
\end{equation}
where the constant $c$ depends on $T,\,\alpha,\,\beta$, and on the domain, while $P$ is
the Leray projection operator on divergence-free and tangential to the boundary vector fields.
\end{theorem}
For our purposes, a very interesting application consists in using this result to find
improved regularity for the solutions of a linear problem in which
$\mathcal{F}=(u\cdot\nabla)\, v$, where $u$ is an assigned function with a given
space-time summability as for instance~\eqref{eq:Shinbrot}.

We have the following lemma.
\begin{lemma}
  \label{lem:maximal-regularity-Oseen}
  Let $\Omega\subset\R^{3}$ be smooth and bounded (or exterior or the full-space or
  even the half-space) and let $u\in L^{r}(0,T;L^{s}(\Omega))$, such that
  $\nabla\cdot u=0$ in $\mathcal{D}'(\Omega)$, for a.e. $t\in [0,T]$. Then, the boundary
  initial value problem
  \begin{equation}
    \label{eq:linear-Stokes2}
    \begin{aligned}
      \partial_{t}v-\Delta v+\nabla \pi&=-(u\cdot\nabla)\, v\qquad &\text{in }(0,T)\times\Omega,
      \\
      \nabla \cdot v&=0 \qquad &\text{in }(0,T)\times\Omega,
      \\
      v(t,x)&=0 \qquad &\text{on }(0,T)\times\partial\Omega,
      \\
      v(0,x)&=0 \qquad &\text{in }\Omega,
    \end{aligned}
  \end{equation}
has a unique solution such that
\begin{equation*}
  v\in L^\infty(0,T;L^2(\Omega))\cap L^2(0,T;H^1_0(\Omega)).
\end{equation*}
Moreover, if the solution $v$ satisfies also 
\begin{equation*}
\nabla v\in L^\alpha(0,T;L^\beta(\Omega))\qquad\text{ for some }1\leq
\alpha,\beta\leq\infty,
\end{equation*}
then the solution $v$ itself satisfies also the following improved estimate 
\begin{equation*}
  \exists\, c>0\quad
  \|\partial_{t}v\|_{L^{\frac{r\alpha}{r+\alpha}}(L^\frac{s\beta}{s+\beta})}+\|P\Delta 
  v\|_{L^{\frac{r\alpha}{r+\alpha}}(L^\frac{s\beta}{s+\beta})}+\|\nabla
  \pi\|_{L^{\frac{r\alpha}{r+\alpha}}(L^\frac{s\beta}{s+\beta})}\leq
  c\,\|u\|_{L^r(L^s)}\|\nabla v\|_{L^\alpha(L^\beta)}, 
\end{equation*}
for some non negative constant $c=c(T,r,s,\alpha,\beta,\Omega)$.
\end{lemma}

In order to show the requested regularity which will make the calculations justified in
the proof of Theorem~\ref{thm:very-weak} we need to apply the above result several (but a
finite number of) times, making possible a sort of bootstrap. This is the main
difference with the proof of the result for $r=s=4$ from~\cite{Gal2018}, where  a single
step was enough. 

In our case we need to apply an extremely
 sharp interpolation result, which has been
kindly provided to us by Prof. Amann~\cite{Ama2018}.
\begin{lemma}
\label{lem:embedding}
Let $\phi\in W^{1,p}(0,T;L^q(\Omega))\cap L^p(0,T;W^{2,q}\cap
W^{1,q}_0(\Omega))$, with $\phi(0)=0$, for $1<p,q<\infty$. Then, it follows that 
%
%
%
%
%
\begin{equation*}
  \phi\in L^{p_{1}}(0,T;W^{1,q}_0(\Omega))\qquad \text{ for all }p_1\leq p_{*},\ \text{ where
  }\frac{1}{p_{*}}:=\frac{1}{p}-\frac{1}{2}.
\end{equation*}
\end{lemma}
\begin{remark}
  We observe that in this embedding we have also compactness for any given $p_1<p_*$ and
  this is a classical, e.g.,  see Simon~\cite[Corollary.~8]{Sim1987}).
 We note that the case $r=s=4$ can be treated directly, without Lemma~\ref{lem:embedding}
 and also we observe that if instead of the sharp continuous embedding one uses the
 classical one (not valid in the endpoint case) the same result of
 Theorem~\ref{thm:very-weak} can be obtained, under the more restrictive condition 
$ u\in L^r(0,T;L^s(\Omega))$  \text{with} $\frac{2}{r}+\frac{2}{s}<1$, \text{for} $s>4$. 

\end{remark}

The interpolation type inequality from Lemma~\ref{lem:embedding} allows us to infer the
following result.
\begin{proposition}
  \label{prop:regularity}
  Let $u\in L^r(0,T;L^s(\Omega))$ be given, with
  $\frac{1}{r}+\frac{1}{s}=\gamma\leq\frac{1}{2}$. If the unique solution of
  \eqref{eq:linear-Stokes2} satisfies $\nabla v\in L^\alpha(0,T;L^\beta(\Omega))$, then in
  addition
  \begin{equation}
    \label{eq:improved-regularity}
    \nabla v\in L^{\alpha_1}(0,T;L^{\frac{\beta s}{\beta+s}}(\Omega))\qquad \forall\,
    \alpha_1\leq\Big(\frac{\alpha\, r}{\alpha+r}\Big)_*.
  \end{equation}
\end{proposition}
We observe that since by definition
$1/\Big(\frac{\alpha\,r}{\alpha+r}\Big)_*=\frac{\alpha+r}{\alpha\, r}-\frac{1}{2}$, then the
following equalities hold true
\begin{equation*}
  \frac{\alpha+r}{\alpha\, r}-\frac{1}{2}+\frac{\beta+s}{\beta\, s}
=\frac{1}{r}+\frac{1}{s}-\frac{1}{2}+
\frac{1}{\alpha}+\frac{1}{\beta}=\gamma-\frac{1}{2}+\frac{1}{\alpha}+\frac{1}{\beta},
\end{equation*}
hence the value $\gamma=\frac{1}{2}$ is exactly the one which keeps the mixed space-time
regularity unchanged. Since we will need this to iterate the procedure to infer further
regularity for the solution of the linear problem, we give now the precise regularity
which is obtained by using repeatedly the same argument.  To this aim let
$(\alpha_1,\beta_1)=(2,2)$ (this because the starting point is the energy space, that is
$\nabla v\in L^2((0,T)\times \Omega)$), we define then by recurrence
\begin{equation*}
  \alpha_{n+1}:=\Big(\frac{\alpha_n\, r}{\alpha_n+r}\Big)_*\qquad \text{and}
  \qquad \beta_{n+1}:=\frac{\beta_n \,s}{\beta_n+s}.
\end{equation*}
We observe that 
\begin{equation}
    \label{eq:iteration}
    \frac{1}{\alpha_{n+1}}+\frac{1}{\beta_{n+1}}=
   \gamma-\frac{1}{2}+ \frac{1}{\alpha_n}+\frac{1}{\beta_n}\qquad \forall\,n\in\N,    
  \end{equation}
  and taking into account that the iteration stops if the limiting values $1$ or $+\infty$
  for either $\alpha_{n+1}$ or $\beta_{n+1}$ are reached.
  \begin{remark}
    By direct computations using the recurrence definition, it follows that
    $\beta_{n+1}\leq \beta_n$, and consequently $\alpha_{n+1}\geq \alpha_n$. Hence, the
    use of the maximal regularity applied to the linear problem decreases the available
    regularity in the space variables, but increases the one of the time variable.
  \end{remark}

Concerning the linear Stokes problem we have the following result
\begin{proposition}
\label{prop:boot-strap}
  Let $u\in L^r(0,T;L^s(\Omega))$ be given with
  $\frac{1}{r}+\frac{1}{s}=\gamma\leq \frac{1}{2}$. Then, the unique solution $v$ of
  \eqref{eq:linear-Stokes2} satisfies
  \begin{equation}
    \label{eq:improved-regularity2}
    \nabla v\in L^{\widetilde{\alpha}_n}(0,T;L^{\beta_n}(\Omega)),
  \end{equation}
  for all $\widetilde{\alpha}_n\leq\alpha_n$, where the couple $(\alpha_n,\beta_n)$ is
  defined as in~\eqref{eq:iteration}. Hence, by H\"older inequality we have also 
\begin{equation}
  \label{eq:eeee}
    (u\cdot\nabla)\, v,\,\nabla \pi\in
    L^{\frac{\widetilde{\alpha}_n \,r}{\widetilde{\alpha}_n+r}}(0,T;L^{\frac{\beta_n \,s}{\beta_n+s}}(\Omega)),
\end{equation}
and
\begin{equation*}
  \frac{1}{{\frac{\widetilde{\alpha}_n\,
        r}{\widetilde{\alpha}_n+r}}}+\frac{1}{{\frac{\beta_n\,s}{\beta_n+s}}}\leq\gamma+1. 
\end{equation*}
\end{proposition}
Observe that, while in the time variable we can consider by H\"older inequality also
exponents smaller than $\alpha_n$, in the space variables, this is not possible when the
domain $\Omega$ is unbounded, and with infinite Lebesgue measure.

In terms of scaling the result of Proposition~\ref{prop:boot-strap} can be viewed as a
modulated and controlled exchange of regularity between space-and time (which resembles
some of the estimates obtainable with Fourier analysis and has connections also with the
early estimates for the gradient of the vorticity in $L^{4/3+\varepsilon}((0,T)\times\Omega)$
from~\cite{Con1990}). The above result shows that the standard theory of linear parabolic
equations can be used to obtain a full scale of results with similar regularity drain from
some variables towards others.  This will be enough for our purposes of applications to
very-weak solutions.
\section{On the energy equality for Leray-Hopf weak solutions.}
\label{sec:weak}
Before proceeding with the proof of Theorem~\ref{thm:grad-ranges}, let us recall that we
mainly need to show that the conditions in the statement of Theorem~\ref{thm:grad-ranges}
are sufficient to show that the double integral $\int_0^T(u\cdot\nabla u,u)\,\dt$ is
finite, hence that~\eqref{eq:fundamental} holds true. The calculations leading
to~\eqref{eq:energy-equality} can be justified by approximating $u$ by the family of
smoother function $(u)_\varepsilon=k_\varepsilon*_t u$ and then taking the limit for
$\varepsilon\to0^+$.  The (time) mollification operator, denoted in the sequel by
$(\cdot)_\varepsilon$, is defined for a space-time function $\Phi:\,(0,T)\times \Omega\to\R^n$ by
\begin{equation}
  \label{eq:mollification}
  (\Phi)_{\varepsilon}(t,x) :=\int_0^{T}
  k_\varepsilon(t-\tau)\Phi(\tau,x)\,\d\tau\qquad \text{for }0<\varepsilon<T . 
\end{equation}
It is a standard Friederichs mollification with respect to the time variable, where $k$ is
a $C^\infty_0(\R)$, real-valued, non-negative even function, supported in $[-1, 1]$ with
$\int_{\R} k(s)\,\ds=1$ and, as usual, $k_\varepsilon(t):=\varepsilon^{-1}k(t/\varepsilon)$.

Before starting the proof we recall the following property of weak solutions, which is a
special case of a result by Hopf~\cite{Hop1951}.
\begin{lemma}
  \label{lem:hopf}
  Let $u_0\in H$ and let $u$ be a Leray-Hopf weak solution of~\eqref{eq:NS}.  Then $u$ can
  be redefined on a set of zero Lebesgue measure in such a way that $u(t)\in L^2(\Omega)$
  for all $t\in[0,T)$ and it satisfies the identity
\begin{equation} 
  \label{eq:Hopf equality}
  (u(s), \phi(s))= (u_0, \phi(0))+\int_0^T \left[\left(u,\frac{\partial\phi}{\partial
        t}\right)- ((u \cdot \nabla)\, u, \phi)- (\nabla u, \nabla \phi)\right]\,\d \tau, 
\end{equation}
for all $\phi \in C^\infty_0([0,T[;C^\infty_0(\Omega))$, with $\nabla\cdot \phi=0$ and all
$0\leq s<T$. 
\end{lemma}

We observe that from the above Lemma it also follows that $u$ is continuous as a map
$[0,T]$ with values in $L^2(\Omega)$, if endowed with the weak topology, that is 
\begin{equation*}
  \forall\,t_0\in[0,T)\qquad  \lim_{t\to t_0}(u(t)-u(t_0),\phi)=0\qquad \forall\,\phi\in
  L^2(\Omega). 
\end{equation*}

\bigskip

The procedure of mollifying and taking the limits clearly explained for instance in
Galdi~\cite[Sec.~4]{Gal2000a}, even if we need to make some changes and adapt to our
setting the results.  Anyway, one has mainly to use the following two facts
\begin{equation}
\label{eq:properties-time-mollifier}
  \begin{aligned}
    & \int_0^T(\nabla u,\nabla (u)_\varepsilon)\,\dt=\int_0^T(\nabla u,(\nabla
    u)_\varepsilon)\,\dt\overset{\varepsilon\to0^+}{\longrightarrow} 
    \int_0^T\|\nabla u\|^2\,\dt,
    \\
    &      (u(t),(u)_\varepsilon(t))=\frac{\|u(t)\|^2}{2}+\mathcal{O}(\varepsilon).
  \end{aligned}
\end{equation}
which come from the choice of the mollifier and the first relies on the fact that $u\in
L^2(0,T;V)$, while the second comes from the weak-$L^2$-continuity of $u$ as a Leray-Hopf
weak solution.

Next, the probably most relevant point is to show that
\begin{equation*}
\int_0^T((u\cdot\nabla)\, u,
    u_\varepsilon)\,\dt\overset{\varepsilon\to0^+}{\longrightarrow}0,
\end{equation*}
under the assumption of the theorem, since this step cannot be deduced --at present--
from the properties of weak solution.

We can now give the proof of the first result of this paper.

\begin{proof}[Proof of Theorem~\ref{thm:grad-ranges}]
  The proof follows some rather standard arguments but relies also on some new estimates
  for the nonlinear term.  Let $0<T<+\infty$ be given and let us fix $t_0$ with $0<t_0\leq
  T$.  Now, let $\{u_m\} \subset C^\infty_0([0,T[;C^\infty_0(\Omega))$ be a sequence
  converging to $u$ in $L^2(0,T; V)\cap L^p(0,T; W^{1,q}_0(\Omega))$, whose existence
  comes from standard density results.  We choose in~\eqref{eq:Hopf equality} $s=t_0$ and
  as legitimate test function $\phi=(u_m)_\varepsilon$, for some $0<\varepsilon<t_0$. In this
  way we get the identity
\begin{equation}
  \begin{aligned}
    \label{eq:regularised_equality}
    (u(t_0), (u_m)_\varepsilon(t_0))&= (u_0, (u_m)_\varepsilon(0))
    \\
    &+\int_0^{t_0} \left[\left(u,\frac{\partial(u_m)_\varepsilon}{\partial t}\right)- ((u
      \cdot \nabla)\, u, (u_m)_\varepsilon)- (\nabla u, (\nabla
      u_m)_\varepsilon)\right]\,\d t.
  \end{aligned}
\end{equation}
Our first goal is to study in the previous equality, taking first the limit for
$m\rightarrow \infty$, with $\varepsilon>0$ fixed, and then the limit as
$\varepsilon\to0^+$.  As usual, in passing to the limit the term which requires more care
is the nonlinear one.  So we focus on the following
\begin{equation*}
  \int_0^{t_0} ((u \cdot \nabla)\, u, (u_m)_\varepsilon)\, \d t=  \int_0^{t_0}  \int_0^{t_0}
  k_\varepsilon (t-\tau) ((u(t) \cdot \nabla)\, u(t), u_m (\tau)) \,\d \tau \dt.
\end{equation*}
For this purpose we split it into three terms as follows
\begin{equation}
  \label{eq:to-prove}
  \begin{aligned}
    \int_0^{t_0} ((u \cdot \nabla)\, u, (u_m)_\varepsilon) \,\d t =& \int_0^{t_0} ((u \cdot
    \nabla)\, u, (u_m)_\varepsilon-(u)_\varepsilon) \,\d t
    \\
    &+\int_0^{t_0} ((u \cdot \nabla)\, u, (u)_\varepsilon-u)\, \d t
    +\int_0^{t_0} ((u \cdot \nabla)\, u, u) \,\d t,
\end{aligned}
\end{equation}
and show that: a) the first term from the right-hand side converges to zero, as
$m\to+\infty$, with $\varepsilon>0$ fixed; b) the second term from the right-hand side
converges to zero, as $\varepsilon\to0^+$; c) the last term one is exactly zero, provided
that $u$ is a weak solutions satisfying one of the conditions of
Theorem~\ref{thm:grad-ranges}. We note that the splitting is a little bit different from
the usual one employed in~\cite{Pro1959,Lio1960,Shi1974}, since the hypotheses are now on
the gradient of the velocity, instead of the velocity itself. For this reason we give
detailed proofs of the steps which are not the same as in the previous papers.

We start by handling the point c) regarding the third term and proving that
\begin{equation} 
  \label{eq:tesi 1}
  \int_0^{t_0} ((u \cdot \nabla)\, u, u) \,\d t=0,
\end{equation}
distinguishing among the three ranges given by Theorem~\ref{thm:grad-ranges}.

Let $\{u_m\} \subset C^\infty_0([0,T[;C^\infty_0(\Omega))$ be a sequence converging to $u$
in the space $L^\infty(0,T;H)\cap L^2(0,T;V)\cap L^p(0,T;\overset{.}{W}^{1,q}(\Omega))$,
where $\overset{.}{W}^{1,q}(\Omega)$ denotes as usual the homogeneous space, endowed with
the semi-norm $\|\nabla u\|_{L^q}$.  Since the field $u_m$ is smooth, integrating by parts
and since is redefined in such a way that $u(t)\in H$ for all $t\in[0,T)$, we get
\begin{equation*}
  \int_0^{t_0} ((u \cdot \nabla)\, u_m, u_m)\,\d t=0.
\end{equation*}
Thus, if we are able to show that 
\begin{equation*}
  \int_0^{t_0} ((u \cdot \nabla)\, u_m, u_m)\,\d t \rightarrow \int_0^{t_0} ((u \cdot
  \nabla)\, u, u)\,\d t, 
\end{equation*}
this would imply~\eqref{eq:tesi 1}.  To this end let us we handle the absolute value of
the difference as follows
\begin{equation*}
  \begin{aligned}
    &\abs{\int_0^{t_0} ((u \cdot \nabla)\, u_m, u_m)\,\d t - \int_0^{t_0} ((u \cdot
      \nabla)\, u, u)\,\d t}
    \\
    &\leq \abs{\int_0^{t_0} ((u \cdot \nabla)\, u_m, (u_m-u))\,\d t} +\abs{ \int_0^{t_0}
      ((u \cdot \nabla)\, (u_m- u), u)\,\d t}.
  \end{aligned}
\end{equation*}
and we will show the convergence of the two terms from the right hand side to zero in the
three different cases.

In the sequel we will use the notation $a\lesssim b$ to denote that there exists a
constant $C>$, depending only on $p,q,\Omega,T$ through the various inequalities valid in
Lebesgue and Sobolev spaces, but not depending on the solution $u$ itself, such that
$a\leq C\,b$.

\bigskip

\begin{itemize}
\item[$(i)$] In the case $\frac{3}{2}< q<\frac{9}{5}$ we have $\nabla u\in
  L^p(0,T;L^q(\Omega))$ with $p=\frac{q}{2q-3}$.  Let us note that in this range we have
  $q^\ast<2q' <6$, where $1/q'+1/q=1$ and $q'$ is the conjugate exponent of $q$.  Hence,
  we can interpolate and write that $\widetilde{q}:=2q'$ satisfies
  $1/\widetilde{q}=\theta/q^\ast+(1-\theta)/6$, which gives us $\theta:= (3-2q)/(3(q-2))$.
  Now, we can estimate, by the Sobolev embedding theorems, the space integral, obtaining
\begin{align*}
  \abs{ \int_0^{t_0} ((u \cdot \nabla)\,(u_m- u), u)\,\d t} &\leq \abs{ \int_0^{t_0}
    \norm{u}_{\widetilde{q}} \norm{\nabla (u_m- u)}_q \norm{u}_{\widetilde{q}}\,\d t}
  \\
  &\leq \abs{ \int_0^{t_0} \norm{u}_{q^\ast}^{2\theta} \norm{
      u}_{6}^{2(1-\theta)}\norm{\nabla (u_m- u)}_q \,\d t}
  \\
  &\lesssim \abs{ \int_0^{t_0} \norm{\nabla u}_{q}^{2\theta} \norm{ \nabla
      u}_{2}^{2(1-\theta)}\norm{\nabla (u_m- u)}_q \,\d t}
  \\
  &\lesssim \norm{\nabla u}^{2\theta}_{p,q} \norm{ \nabla u}^{2(1-\theta)}_{2, 2}
  \norm{\nabla (u_m- u)}_{p,q},
\end{align*}
where we have applied in the last step H\"older inequality in the time variable with
exponents $x,y,z$ such that $2(1-\theta)x=2$, $z=p$, and $2\theta y=p$, and thus
satisfying $1/x+1/y+1/z=1$.  This shows that $\int_0^{t_0} ((u \cdot \nabla)\,(u_m-u),
u)\,\d t\rightarrow 0$, as $m\to+\infty$.

Analogously, we have also the following estimates
\begin{align*}
  &\abs{ \int_0^{t_0} ((u \cdot \nabla)\,u_m ,(u_m-u))\,\d t}
  \\
  &\leq \abs{ \int_0^{t_0} \norm{u}_{\widetilde{q}} \norm{\nabla u_m}_q
    \norm{u_m-u}_{\widetilde{q}}\,\d t}
  \\
  &\leq\abs{ \int_0^{t_0} \norm{u}_{q^\ast}^{\theta} \norm{
      u}_{6}^{(1-\theta)}\norm{\nabla u_m}_q \norm{u_m-u}_{q^\ast}^{\theta} \norm{
      u_m-u}_{6}^{(1-\theta)}\,\d t}
  \\
  &\lesssim \abs{ \int_0^{t_0} \norm{\nabla u}_{q}^{\theta} \norm{ \nabla
      u}_{2}^{(1-\theta)}\norm{\nabla u_m}_q \norm{\nabla (u_m-u)}_{q}^{\theta} \norm{
      \nabla (u_m-u)}_{2}^{(1-\theta)}\d t}
  \\
  &\lesssim \norm{\nabla u}^\theta_{p,q} \norm{ \nabla u}^{(1-\theta)}_{2, 2} \norm{\nabla u_m}_{p,q} \norm{\nabla
    (u_m-u)}^\theta_{p,q} \norm{ \nabla (u_m-u)}^{(1-\theta)}_{2, 2},
\end{align*}
where we used H\"older inequality in the time variable with exponents $\tilde{x}=2x$ for
both the terms $\norm{ \nabla u}_{2}^{(1-\theta)}$ and $\norm{ \nabla
  (u_m-u)}_{2}^{(1-\theta)}$, with $\tilde{y}=2y$ for both the terms $\norm{\nabla
  u}_{q}^{\theta}$ and $\norm{\nabla (u_m-u)}_{q}^{\theta}$, and with $\tilde{z}=z$ for
the term $\norm{\nabla u_m}_q$.  This implies that, as $m\to+\infty$ it holds
$\int_0^{t_0} ((u \cdot \nabla)\, u_m, (u_m-u))\,\d t\rightarrow 0$. 

\bigskip

The case $(ii)$ is better handled when split into  two sub-cases.

\medskip

\item[$(ii)_1$] In the case $\frac{9}{5}\leq q <\frac{12}{5}$ we have $\nabla u\in
  L^p(0,T; L^q(\Omega))$ with $p=\frac{5q}{5q-6}$.  Let us note that in this range we have
  $2<2q'<q^\ast $, where $1/q'+1/q=1$.  Hence, we can interpolate and write that
  $\widetilde{q}:=2q'$ satisfies $1/\widetilde{q}=\theta/2+(1-\theta)/q^\ast$, which
  implies $\theta:= (5q-9)/(5q-6)$.  Now, by using the Sobolev embedding theorems, we show
  that
\begin{align*}
  \abs{ \int_0^{t_0} ((u \cdot \nabla)\,(u_m- u), u)\,\d t} &\leq \abs{ \int_0^{t_0}
    \norm{u}_{\widetilde{q}} \norm{\nabla (u_m- u)}_q \norm{u}_{\widetilde{q}}\,\d t}
  \\
  &\leq \abs{ \int_0^{t_0} \norm{u}_{2}^{2\theta} \norm{
      u}_{q^\ast}^{2(1-\theta)}\norm{\nabla (u_m- u)}_q \,\d t}
  \\
  &\lesssim \abs{ \int_0^{t_0} \norm{u}_{2}^{2\theta} \norm{ \nabla
      u}_{q}^{2(1-\theta)}\norm{\nabla (u_m- u)}_q \,\d t}
  \\
  &\lesssim \norm{u}_{\infty,2}^{2\theta} \norm{ \nabla u}^{2(1-\theta)}_{p, q} \norm{\nabla (u_m-
    u)}_{p,q},
\end{align*}
thanks to H\"older inequality in the time variable with exponents $\gamma_1,\gamma_2$ such
that $2(1-\theta)\gamma_1=p$, $\gamma_2=p$, with $1/\gamma_1+1/\gamma_2=1$.  This shows
that $ \int_0^{t_0} ((u \cdot \nabla)\,(u_m- u), u)\,\d t\rightarrow 0$. The term
$\abs{\int_0^{t_0} ((u \cdot \nabla)\,u_m, (u_m-u))\,\d t}$ can be treated in a very
similar way, so that to conclude that $\int_0^{t_0} ((u \cdot \nabla)\,u_m, u_m)\,\d
t\rightarrow \int_0^{t_0} ((u \cdot \nabla)\,u, u)\,\d t$. 

\bigskip

\item[$(ii)_2$] In the case $\frac{12}{5}\leq q \leq 3$ we have $\nabla u\in L^p(0,T;
  L^q(\Omega))$ with $p=\frac{5q}{5q-6}$.  Let us denote by $r$ the exponent such that
  $1/r+1/q=1/2$.  Let us note that in this range we have $r<q^\ast$.  Hence, we can use
  the convex interpolation between $2$ and $q^*$ to write that $r$ satisfies
  $1/r=\theta/2+(1-\theta)/q^\ast$, which gives us $\theta:= (5q-12)/(5q-6)$.  Now,  by
  the Sobolev embedding theorems, we can
  estimate the integral as follows
\begin{align*}
  \abs{ \int_0^{t_0} ((u \cdot \nabla)\,(u_m- u), u)\,\d t} &\leq \abs{ \int_0^{t_0}
    \norm{u}_{2} \norm{\nabla (u_m- u)}_q \norm{u}_{r}\,\d t}
  \\
  &\leq \abs{ \int_0^{t_0} \norm{u}_{2}^{1+\theta} \norm{\nabla (u_m- u)}_q
    \norm{u}_{q^\ast}^{1-\theta}\,\d t}
  \\
  &\lesssim \abs{ \int_0^{t_0} \norm{u}_{2}^{1+\theta} \norm{\nabla (u_m- u)}_q \norm{\nabla
      u}_{q}^{1-\theta}\,\d t}
  \\
  &\lesssim \norm{u}_{\infty,2}^{1+\theta} \norm{\nabla (u_m- u)}_{p,q} \norm{ \nabla
    u}^{1-\theta}_{p, q},
\end{align*}
where we have applied H\"older inequality with respect to the time variable with exponents
$\alpha_1, \alpha_2$ such that $(1-\theta)\alpha_1=p$ and $\alpha_2=p$, with
$1/\alpha_1+1/\alpha_2=1$.  For the the other term we can argue similarly and thus this
shows that $\int_0^{t_0} ((u \cdot \nabla)\,u_m, u_m)\,\d t \rightarrow \int_0^{t_0} (u
\cdot\nabla u, u)\,\d t$ in this case, as well.

The case $q=3$ can be treated in a mostly similar way, by observing that
    $W^{1,3}(\Omega)\subset L^r(\Omega)$, for all finite $3\leq r$, while the rest of the
    proof remains unchanged.

\bigskip

\item[$(iii)$] In case $q>3$ we have $\nabla u\in L^p(0,T; L^q(\Omega))$ with
  $p=1+\frac{2}{q}$.  Let us denote by $q'$ the conjugate exponent of $q$,
  i.e. $1/q'+1/q=1$ and by $1<\widetilde{q}:=2q'<\infty$.  We estimate the
  $L^{\widetilde{q}}(\Omega)$-norm from convex interpolation between $L^2(\Omega)$ and
  $L^\infty(\Omega)$, with $1/\widetilde{q}=\theta/2$, which gives us
  $\theta:=1-\frac{1}{q}$.  Now, by Sobolev embedding theorems
\begin{align*}
\abs{ \int_0^{t_0} ((u \cdot \nabla)\,(u_m- u), u)\,\d t} &\leq \abs{ \int_0^{t_0}
  \norm{u}_{\widetilde{q}} \norm{\nabla (u_m- u)}_q \norm{u}_{\widetilde{q}}\,\d t}
\\
&\lesssim \abs{ \int_0^{t_0} \norm{u}_{2}^{2\theta} {\norm{
    u}_{\infty}^{2(1-\theta)}}\norm{\nabla (u_m- u)}_q \,\d t} 
\\
&\lesssim \abs{ \int_0^{t_0} \norm{u}_{2}^{2\theta} {\norm{ \nabla
    u}_{q}^{2(1-\theta)}}\norm{\nabla (u_m- u)}_q \,\d t} 
\\
&\lesssim \norm{ u}^{2\theta}_{\infty, 2}\norm{\nabla u}^{2(1-\theta)}_{p,q}  \norm{\nabla (u_m- u)}_{p,q},
\end{align*}
the last inequality being justified by use of the H\"older inequality in the time variable
with exponents $\beta_1,\beta_2$ such that $2(1-\theta)\beta_1=p$, and $\beta_2=p$, such
that $1/\beta_1+1/\beta_2=1$.  As in the previous cases, we can treat also the other term
in a similar way and conclude that that $\int_0^{t_0} ((u \cdot \nabla)\,u_m, u_m)\,\d t
\rightarrow \int_0^{t_0} (u \cdot\nabla u, u)\,\d t$, as $m\to+\infty$.
\end{itemize}

\medskip 

Collecting the above estimates we conclude the proof of~\eqref{eq:tesi 1}, which
is the most crucial step. The other terms will be treated below, but their analysis
follows more classical path-lines, already outlined in previous works.

\bigskip

Next, for the sake of completeness, we show that the first two terms from the right hand
side of~\eqref{eq:to-prove} converge to zero.  This follows easily using the usual
properties of mollifiers together with similar estimates as above, distinguishing among
the three different ranges.  In case $(i)$, going through the same lines of the preceding
estimates and with same notation, indeed we can bound the first integral from the right
hand side %
as follows
\begin{align*}
  &\hspace{-1cm}\abs{ \int_0^{t_0} ((u \cdot \nabla)\,u, (u_m)_\varepsilon-(u)_\varepsilon)\,\d t} 
  \\
  &\leq {\int_0^{t_0}\norm{u}_{\widetilde{q}} \norm{\nabla u}_q
    \norm{(u_m)_\varepsilon-(u)_\varepsilon}_{\widetilde{q}}\,\d t} 
  \\
  &\leq { \int_0^{t_0} \norm{u}_{q^\ast}^{\theta} \norm{u}_{6}^{1-\theta} 
    \norm{\nabla u}_q  \norm{(u_m)_\varepsilon-(u)_\varepsilon}_{q^\ast}^{\theta}
    \norm{(u_m)_\varepsilon-(u)_\varepsilon}_{6}^{1-\theta} \,\d t}
  \\
  &\lesssim { \int_0^{t_0} \norm{\nabla u}_{q}^{1+\theta} \norm{ \nabla u}_{2}^{1-\theta}
    \norm{\nabla ((u_m)_\varepsilon-(u)_\varepsilon)}_q^{\theta} \norm{\nabla
      ((u_m)_\varepsilon-(u)_\varepsilon)}_2^{1-\theta}\,\d t}
  \\
  &\lesssim \norm{\nabla u}^{1+\theta}_{p,q} \norm{ \nabla u}^{1-\theta}_{2, 2} \norm{\nabla
    ((u_m)_\varepsilon-(u)_\varepsilon)}^{\theta}_{p,q} \norm{\nabla
    ((u_m)_\varepsilon-(u)_\varepsilon)}^{1-\theta}_{2,2}, 
\end{align*}
where we have applied H\"older inequality in the time variable with exponents
$\eta_1,\eta_2,\eta_3$ such that $(1+\theta)\eta_1=p$, $(1-\theta)\eta_2=2$, and $\theta
\eta_3=p$, and satisfying $1/\eta_1+2/\eta_2+1/\eta_3=1$.  Now, Lemma~2.5
in~\cite{Gal2011} ensures that
\begin{equation*}
  \lim_{m\rightarrow \infty}\norm{\nabla ((u_m)_\varepsilon-(u)_\varepsilon)}_{p,q}=0
  \qquad\text{and}\qquad 
  \lim_{m\rightarrow \infty} \norm{\nabla ((u_m)_\varepsilon-(u)_\varepsilon)}_{2,2}=0, 
\end{equation*}
which, together with the previous estimate, implies that for each fixed $\varepsilon>0$ it holds.
\begin{equation*}
  \lim_{m\to+\infty}\int_0^{t_0} ((u \cdot \nabla)\,u, (u_m)_\varepsilon-(u)_\varepsilon)\,\d t =0.
\end{equation*}
Cases $(ii)$ and $(iii)$ can be treated similarly and we skip the proof which is now
straightforward.

The very same approach allows to bound the integral $\int_0^{t_0} ((u \cdot
\nabla)\,u,(u)_\varepsilon-u)\,\d t $ as follows

\begin{align*}
  &\hspace{-1cm}\abs{ \int_0^{t_0} ((u \cdot \nabla)\,u, (u)_\varepsilon-u)\,\d t} 
\lesssim \norm{\nabla u}^{1+\theta}_{p,q} \norm{ \nabla u}^{1-\theta}_{2, 2} \norm{\nabla
    (u)_\varepsilon-u}^{\theta}_{p,q} \norm{\nabla
    (u)_\varepsilon-u)}^{1-\theta}_{2,2}, 
\end{align*}
which converge to zero as $\varepsilon\to0^+$ once more thanks to
Lemma~2.5~in~\cite{Gal2011}.

By collecting all results we finally get  that
\begin{equation*}
\lim_{\varepsilon\to0^+}\lim_{m\to+\infty}\int_0^{t_0} ((u \cdot \nabla)\,u,
(u_m)_\varepsilon)\,\d t=0. 
\end{equation*}
Passing to the limit as $m\to+\infty$ in \eqref{eq:regularised_equality} we then obtain
\begin{equation*}
  \begin{aligned}
    (u(t_0), (u)_\varepsilon(t_0))&= (u_0, (u)_\varepsilon(0))
+\int_0^{t_0} \left[\left(u,\frac{\partial(u)_\varepsilon}{\partial t}\right)- ((u
      \cdot \nabla)\, u, (u)_\varepsilon)- (\nabla u, (\nabla
      u)_\varepsilon)\right]\,\d t.
  \end{aligned}
\end{equation*}
and we now  let $\varepsilon \rightarrow 0^+$ and the convective term vanishes.  The term
involving the time derivative of $k_\varepsilon$, i.e. $\int_0^{t_0}
\left(u,\frac{\partial(u_m)_\varepsilon}{\partial t}\right)$ vanishes identically since
$k_\varepsilon$ is even. The usual properties of mollifiers from
\eqref{eq:properties-time-mollifier} imply that
in the limit $\varepsilon\to0^+$ we get
\begin{equation*}
  \norm{u(t_0)}^2+ 2\int_0^{t_0} \norm{\nabla u(s)}^2\,\d s = \norm{u_0}^2,
\end{equation*}
which is~\eqref{eq:energy-equality} for $t=t_0$. By the arbitrariness of $t_0$, the proof
of Theorem~\ref{thm:grad-ranges} is complete.
\end{proof}
\section{On the energy equality for distributional solutions.}
\label{sec:very-weak}
In this section we prove Theorem~\ref{thm:very-weak} concerning the energy equality for a
class of solution which do not satisfy a priori the energy inequality. The proof follows
by a duality argument as done in~\cite{Gal2018}, who treated the case $r=s=4$. Here, we
cover the full range of exponents by following a similar technique, but performing a more
detailed analysis of the regularity of the solution of the adjoint problem. We point out
that also in our case, the main point is to show that a very-weak solution (in the sense
of Definition~\ref{def:very-weak}) with initial datum in $H$ and such
that~\eqref{eq:Shinbrot} holds true becomes a Leray-Hopf weak solution, hence satisfies
the energy equality by the classical results from~\cite{Shi1974}. From the proof and from
the analysis performed in Sec.~\ref{sec:comparison-very-weak} it turns out that less
stringent hypotheses (as those considered in the Theorem~\ref{thm:grad-ranges}) are not
sufficient to conclude the proof, hence condition~\eqref{eq:Shinbrot} seems not improvable
for very weak solutions, at present. 

By following the same notation from~\cite{BG2004,Gal2018}, given
$g:(0,T)\times\Omega\to\R^3$ we define the function $\widetilde{g}$ as follows
\begin{equation*}
  \widetilde{g}(t,x):=g(T-t,x).
\end{equation*}

In the sequel we will need again to approximate in a suitable way functions defined on
$(0,T)\times\Omega\to\R^3$, mollifying also in the space variables. 
\begin{remark}
  In the case of the Cauchy problem the smoothing can be done simply by mollifying also in
  the space variables (as in \cite{Gal2018}) while in the case of an exterior domain one
  can use practically the same approach we propose, by using the result by Borchers and
  Sohr~\cite{BS1990} in the appropriate setting.
\end{remark}
For this purpose we define
\begin{equation*}
  \Omega_\varepsilon:=\left\{x\in\Omega\subset \R^3: \ |x|<\frac{1}{\varepsilon},\
    d(x,\partial\Omega)\geq2\varepsilon\right\},  
\end{equation*}
with $d(x,A)$ the Euclidean distance from $x\in \R^3$ to the closed set $A\subset \R^3$.
For small enough $\varepsilon>0$, the set $\Omega_\varepsilon$ turns out to be bounded and
non-empty. Then, by taking $\rho\in C^\infty_0(\R^3)$, non-negative, radial, and such that
$\int_{\R^3}\rho(x)\,dx=1$, we define as usual the scaled
$\rho_\varepsilon(x)=\varepsilon^{-3}\rho(x \varepsilon^{-1})$. Moreover, if $\chi_A$
denotes the indicator function of the measurable set $A\subset\R^3$, we define, for any
$v\in (L^1_{loc}(\Omega))^3$, the function $^{\varepsilon}v(x)\in (C^\infty_0(\Omega))^3$ as follows
\begin{equation*}
^{\varepsilon}v(x):=[(v\ \chi_{\Omega_\varepsilon})*\rho_\varepsilon](x).
\end{equation*}
It follows that $^{\varepsilon}v$ is smooth and with compact support by the standard
  properties of mollifiers, but $\nabla\cdot (^{\varepsilon}v)\not=0$. For the energy
  estimates used later on we need a divergence-free approximation, hence one possibility
  is to use the Bogovski\u{\i}~\cite{Bog1980} operator $\mathcal{B}[f]$ defined as follows
  \begin{equation*}
          \mathcal{B}[f](x):=\int_{\Omega}f(y)\left[\frac{x-y}{|x-y|^{n}}\int_{|x-y|}^{+\infty}\psi\left(y+
          \xi\frac{x-y}{|x-y|}\right)\xi^{n-1}d\xi\right]\,dy, 
  \end{equation*}
  where $\psi\in C^\infty_0(\R^3)$, with support in the unit ball and
  $\int_{\R^3}\psi(y)\,dy=1$, when $\Omega$ contains the unit ball. (Some changes with
  scaling and translations are needed if this is not satisfied).

  It is well-known that if $\Omega$ is smooth and bounded and if $f\in
  C^\infty_0(\Omega)$, then $\mathcal{B}[f]\in C^\infty_0(\Omega)$ and
\begin{equation*}
  \nabla\cdot \mathcal{B}[f](x)=f(x)-\frac{1}{|\Omega|}\int_\Omega f(y)\,\d y.
\end{equation*}
In addition, if $f\in L^s(\Omega)$, with $1<s<\infty$, then $\mathcal{B}[f]\in
W^{1,s}_0(\Omega)$ and moreover it holds (for fixed $\psi$) that $\|f\|_{W^{1,s}}\leq
C(s,\Omega)\|f\|_{L^s}$, see for instance the review in Galdi~\cite{Gal2011}.  We thus
define,
\begin{equation*}
  v^{(\varepsilon)}(x):= {^{\varepsilon}v}(x)-\mathcal{B}[\nabla\cdot(^{\varepsilon}v)](x),
\end{equation*}
and since $\int_\Omega \nabla\cdot (^{\varepsilon}v)\,\d x=\int_{\partial\Omega}
{^{\varepsilon}v}\cdot n\,\d S=0$, it follows that if $v\in L^s(\Omega)$, then we have the
following relevant properties
\begin{itemize}
\item[i)] $v^{(\varepsilon)}(x) \in C^\infty_0(\Omega)$, with $\nabla \cdot
  v^{(\varepsilon)}=0$;
\item[ii)] $\|v^{(\varepsilon)}(x)\|_{L^s}\leq C(s,\Omega)\|v\|_{L^s}$, uniformly in $\varepsilon>0$;
\item[iii)] $v^{(\varepsilon)}\to v \text{ in }L^s(\Omega)$, as $\varepsilon\to0^+$.
\end{itemize}

We apply this procedure to the initial datum $u_0\in L^2(\Omega)$ and we define the family
of functions $\{u^{(\varepsilon)}_0\}_{\varepsilon>0}\subset C^\infty_0(\Omega)$ in such a
way that
\begin{equation*}
\nabla\cdot u_0^{(\varepsilon)}=0\qquad\text{ and also }\qquad
\lim_{\varepsilon\to0^+}\|u_0^{(\varepsilon)}-u_0\|=0. 
\end{equation*}
We also need to apply a similar smoothing procedure to space-time functions. Hence, given
$u\in L^r(0,T;L^s(\Omega))$ we define $u^{(\varepsilon)}$ with the above technique
obtaining a family $u^{(\varepsilon)}$ of functions which are divergence-free and smooth
in the space variables. To regularize also in the time variable we use the same
Friederichs mollifier as introduced in~\eqref{eq:mollification} and we define, for
$0<\epsilon<T$ 
\begin{equation*}
  u_{(\varepsilon)}(t,x):=\int_0^{T} k_\varepsilon(t-\tau) u^{(\varepsilon)}(\tau,x)\,\d\tau,
\end{equation*}
in such a way that 
\begin{equation}
  \label{eq:smoothed-u}
  \begin{aligned}
    i)&\quad u_{(\varepsilon)}(t,x) \in L^\infty((0,T)\times \Omega),
    \\
    ii)&\quad\nabla \cdot
    u_{(\varepsilon)}=0,
    \\
iii)&\quad u_{(\varepsilon)}\to u \text{ in }L^r(0,T;L^s(\Omega)),\text{ as }\varepsilon\to0^+.
  \end{aligned}
\end{equation}
It turns out that $u_{(\varepsilon)}$ and also $\widetilde{u_{(\varepsilon)}}$ (that one
with reversed time) are then suitable' vector fields to be used as a convection term in
Oseen-type systems, in order to obtain global solution in the energy space, that is in
$L^\infty(0,T;H)\cap L^2(0,T;V)$.

With this regularization at disposal, we can in fact start by considering
$(u^{\varepsilon},p^\varepsilon)$ solution of the linear Oseen problem
  \begin{equation}
    \label{eq:adjoint1}
    \begin{aligned}
      \partial_{t}u^\varepsilon-\Delta
      u^\varepsilon+({u_{(\varepsilon)}}\cdot\nabla)\,u^\varepsilon+\nabla
      p^\varepsilon&=0\qquad \text{in }(0,T)\times\Omega,
      \\
      \nabla \cdot u^\varepsilon&=0 \qquad \text{in }(0,T)\times\Omega,
      \\
      u^\varepsilon(t,x)&=0 \qquad \text{on }(0,T)\times\partial\Omega,
      \\
      u^\varepsilon(0,x)&=u_0^{(\varepsilon)} \qquad \text{in }\Omega,
    \end{aligned}
  \end{equation}
  and by the standard Galerkin method (combined with the method of invading domains if the
  domain is unbounded), it turns out that there exists a unique (distributional) solution 
  $(u^\varepsilon,p^\varepsilon)$ such that
\begin{equation*}
  u^\varepsilon\in W^{1,2}(0,T;H)\cap L^2(0,T;H^2\cap V)\subset
  C(0,T;H^1_0(\Omega))\qquad\text{and}\qquad \nabla p^\varepsilon\in L^2((0,T)\times\Omega). 
\end{equation*}
The main point is that $\big(({u_{(\varepsilon)}}\cdot\nabla)\,u^\varepsilon,
u^\varepsilon)\big)=0$ and $(\nabla p^\varepsilon,u^\varepsilon)=0$, coming from the
divergence free constraint on both $u_{(\varepsilon)}$ and $u^\varepsilon$; hence, one can
show that in the sense of $\mathcal{D}'(0,T)$
\begin{equation*}
  \frac{1}{2}\frac{d}{dt}\|u^\varepsilon\|^2+\nu\|\nabla u^\varepsilon\|^2=0,
\end{equation*}
and from this one can prove a first a priori estimate and then use the standard
theory for linear equation systems of Stokes type to prove existence and regularity of
the solution.

By the definition of very-weak solution we obtain that the following identity is satisfied
by the difference $u-u^\varepsilon$
\begin{equation}
  \label{eq:difference}
  \int_0^T(u-u^\varepsilon,\partial_t \phi+\Delta\phi+(u_{(\varepsilon)}\cdot\nabla)\,\phi)\,\d
  s=\int_0^T((u-u_{(\varepsilon)})\cdot\nabla\,\phi,u)-(u_0-u_0^{(\varepsilon)},\phi(0)),
\end{equation}
for all $\phi\in \mathcal{D}_T$.

Let us consider now, for a given $f\in C^\infty_0((0,T)\times \Omega)$ the solution
$(w^\varepsilon,\pi^\varepsilon)$ of the linear Oseen problem
  \begin{equation}
    \label{eq:adjoint2}
    \begin{aligned}
      \partial_{t}w^\varepsilon-\Delta w^\varepsilon-(\widetilde{u_{(\varepsilon)}}\cdot\nabla)\,w^\varepsilon+\nabla
      \pi^\varepsilon&=-\widetilde{f}\qquad \text{in }(0,T)\times\Omega, 
      \\
      \nabla \cdot w^\varepsilon&=0 \qquad \text{in }(0,T)\times\Omega,
      \\
      w^\varepsilon(t,x)&=0 \qquad \text{on }(0,T)\times\partial\Omega,
      \\
      w^\varepsilon(0,x)&=0 \qquad \text{in }\Omega,
    \end{aligned}
  \end{equation}
and define 
\begin{equation*}
  \Psi^\varepsilon(t,x):=\widetilde{w^\varepsilon}(t,x)={w^\varepsilon}(T-t,x)\qquad\text{and}\qquad
  \Xi^\varepsilon(t,x):=\widetilde{\pi^\varepsilon}(t,x)={\pi^\varepsilon}(T-t,x),
\end{equation*}
in such a way that they have the same regularity as $(w^\varepsilon,\pi^\varepsilon)$ and they solve the
\textit{final value problem}
  \begin{equation}
    \label{eq:final}
    \begin{aligned}
      \partial_{t}\Psi^\varepsilon+\Delta
      \Psi^\varepsilon+(u_{(\varepsilon)}\cdot\nabla)\,\Psi^\varepsilon-\nabla \Xi^\varepsilon &=f\qquad
      \text{in }(0,T)\times\Omega, 
      \\
      \nabla \cdot \Psi^\varepsilon&=0 \qquad \text{in }(0,T)\times\Omega,
      \\
     \Psi^\varepsilon(t,x)&=0 \qquad \text{on }(0,T)\times\partial\Omega,
      \\
      \Psi^\varepsilon(T,x)&=0 \qquad \text{in }\Omega.
    \end{aligned}
  \end{equation}
  By the standard theory of Galerkin methods it follows again that
 \begin{equation*}
   \Psi^\varepsilon\in W^{1,2}(0,T;H)\cap L^2(0,T;H^2\cap V)\subset
   C(0,T;H^1_0(\Omega))\qquad\text{and}\qquad  \nabla \Xi^\varepsilon\in L^2((0,T)\times
   \Omega),  
\end{equation*}
and we would like to use $\Psi^\varepsilon$ as test function in~\eqref{eq:difference},
instead of $\phi\in \mathcal{D}_T$. This can be done by using standard density arguments
to approximate $\Psi^\varepsilon$ in $\mathcal{D}_T$ with uniform bounds in $\epsilon>0$,
provided that $\Psi^\varepsilon$ is smooth enough independently of $\epsilon$ to make all
integrals well-defined.
\begin{proof}[Proof of Theorem~\ref{thm:very-weak}]
  The proof follows by considering \eqref{eq:difference} with test function given by
  $\Psi^\varepsilon$. Formally (calculations would be rigorous by using a smooth
  approximation of $\Psi^\varepsilon$ and this can be constructed by standard density
  arguments, as recalled in the appendix of \cite{Gal2018}) we obtain that
\begin{equation}
  \label{eq:difference2}
  \int_0^T(u-u^\varepsilon,f+\nabla\Xi^\varepsilon)\,\d s
=\int_0^T((u-u_{(\varepsilon)}\cdot\nabla)\,\Psi^\varepsilon,u)\,\d s
-(u_0-u_0^{(\varepsilon)},\Psi^\varepsilon(0)), 
\end{equation}
but the argument will be completely rigorous after we have proved that
$(\Psi^\varepsilon,\Xi^\varepsilon)$ is smooth enough. This will come by a proper use of
the maximal regularity arguments for the Stokes system and by standard density arguments

First, we need to prove that the identity
\begin{equation*}
    \int_0^T(u-u^\varepsilon,\nabla\Xi^\varepsilon)\,\d s=0,
\end{equation*}
is valid, showing enough regularity on $\Xi^\varepsilon$ and using the fact that $\nabla\cdot
(u-u^\varepsilon)=0$ in $\mathcal{D}'(\Omega)$ a.e $t\in(0,T)$.  Next, by taking the limit
as $\varepsilon\to0^+$, the uniform bound in $L^2(\Omega)$ of $\Psi^\varepsilon(0)$
implies that
\begin{equation*}
  \lim_{\varepsilon\to0^+}(u_0-u_0^{(\varepsilon)},\Psi^\varepsilon(0))=0.
\end{equation*}
Since $u\in L^r(0,T;L^s(\Omega))$, then $u-u_{(\varepsilon)}\to0$ in
$L^r(0,T;L^s(\Omega))$ and
\begin{equation*}
  \lim_{\varepsilon\to0^+}\int_0^T((u-u_{(\varepsilon)}\cdot\nabla)\,\Psi^\varepsilon,u)\,\d
  s=0,
\end{equation*}
provided that $\big|\int_0^T((u-u_{(\varepsilon)}\cdot\nabla)\,\Psi^\varepsilon,u)\,\d
s\big|<\infty$, and it is at this point that we will use the
assumption~\eqref{eq:Shinbrot} and its consequences on the regularity of the Oseen system.

In particular, we have to show that 
\begin{equation*}
 \nabla \Xi^\varepsilon,\ ((u-u_{(\varepsilon)}\cdot\nabla)\,\Psi^\varepsilon\in L^{r'}(0,T;L^{s'}(\Omega))
\end{equation*}
or equivalently, by recalling the link between $r$ and $s$ from~\eqref{eq:Shinbrot},
\begin{equation}
  \label{eq:final_estimate}
  \nabla \Xi^\varepsilon,\ ((u-u_{(\varepsilon)}\cdot\nabla)\,\Psi^\varepsilon\in
  L^{\frac{2s}{2+s}}(0,T;L^{\frac{s}{s-1}}(\Omega)), 
\end{equation}
the second coming from the condition $\nabla\Psi^\varepsilon\in
L^{\frac{s}{2}}(0,T;L^{\frac{s}{s-2}}(\Omega))$ by using the H\"older inequality.

Once the above steps are done, one can approximate $ \Xi^\varepsilon$ and
$\Psi^\varepsilon$ by a family of smooth compactly supported functions
$\Theta^{\varepsilon}_{\eta}\in \mathcal{D}_{T}$ and $\Upsilon^{\varepsilon}_{\eta}\in
C^{\infty}_{0}((0,T)\times\Omega)$, with uniform bounds in $\eta>0$ in the spaces
$L^{r'}(0,T;W^{1,s'}(\Omega))$ and $L^{\frac{s}{2}}(0,T;W^{1,\frac{s}{s-2}}(\Omega))$,
respectively. These smooth functions $\Theta^{\varepsilon}_{\eta}$ and
$\Upsilon^{\varepsilon}_{\eta}$, can be then employed as legitimate test functions in the
definition of very-weak solution. We do not explicitly write this since it will be clear
to the reader how to proceed, hence we simply show the requested regularity.

Once this step is performed, the uniform bounds of $u^\varepsilon$ in $L^\infty(0,T;H)\cap
L^2(0,T;V)$ imply --by standard results about weak compactness-- the existence of $U\in
L^\infty(0,T;H)\cap L^2(0,T;V)$ such that, along some sub-sequence $\{\varepsilon_n\}$
converging to $0^+$ as $n\to+\infty$
\begin{equation*}
  \begin{aligned}
    & u^{\varepsilon_n}\overset{*}{\rightharpoonup} U\qquad\text{weakly* in }
    L^\infty(0,T;H)\subset L^\infty(0,T;L^2(\Omega)),
    \\
    & u^{\varepsilon_n}\rightharpoonup U\qquad\text{weakly in } L^2(0,T;V)\subset
    L^2(0,T;H^1_0(\Omega)).
  \end{aligned}
\end{equation*}
hence the taking the limit as $n\to+\infty$ in~\eqref{eq:difference2} considered with
$\varepsilon=\varepsilon_n$ we obtain
\begin{equation*}
\int_0^T(u-U,f)\,\d
s\overset{n\to+\infty}{\longleftarrow}\int_0^T(u-u^{\varepsilon_n},f)\,\d s
\overset{n\to+\infty}{\rightarrow}0 
  \qquad\qquad \forall\,f\in C^\infty_0((0,T)\times\Omega),
\end{equation*}
hence $\int_0^T(u-U,f)\,\d s=0$ and, by the arbitrariness of the smooth $f$, this shows
that $u$ can be identified with $U \in L^\infty(0,T;L^2(\Omega))\cap
L^2(0,T;H^1_0(\Omega))$. This proves that $u$ falls into the class of Leray-Hopf weak
solutions, for which the classical results from~\cite{Shi1974} imply the energy equality.

\bigskip

In order to conclude the proof we need now to justify the method and in particular the
fact that the integrals involving $\Psi^\varepsilon$ and $\Xi^\varepsilon$ are
well-defined, with the available regularity of $u$ from the
assumption~\eqref{eq:Shinbrot}. To this end, observe that the starting point of our
analysis is the fact that $\nabla w^\varepsilon\in L^2(0;T,L^2(\Omega))$ and also $u\in
L^r(0,T; L^s(\Omega))$ is a very-weak solution of~\eqref{eq:NS} with~\eqref{eq:Shinbrot}.

By using Propositions~\ref{prop:regularity} and \ref{prop:boot-strap}, we observe that
under the hypotheses of \eqref{eq:Shinbrot} (that is $\gamma=1/2$) it follows that
$(\alpha_0,\beta_0)=(2,2)$ and the $n$-th iteration of the maximal regularity results from
Proposition~\ref{prop:boot-strap} shows that, for $n\geq1$
\begin{equation*}
  (\widetilde{u_{(\varepsilon)}}\cdot\nabla)\,w^\varepsilon,\nabla \pi^\varepsilon\in
  L^{\alpha_n}(0,T;L^{\beta_n}(\Omega))\qquad\text{with
  }\beta_n=\frac{\beta_{n-1}\,s}{\beta_{n-1}+s}\qquad \text{and
  }\frac{1}{\alpha_n}+\frac{1}{\beta_n}=\frac{3}{2}, 
\end{equation*}
which reduces, in terms of $\Phi^\varepsilon,\Xi^\varepsilon$, and more explicitly in
terms of the parameter $s>4$, to 
\begin{equation*}
  \alpha_{n}={\frac{s}{s-n}}\qquad\text{and}\qquad \beta_{n}={\frac{2s}{2n+s}}
\end{equation*}
and to have then 
\begin{equation}
  \label{eq:aaaaaaa}
  (\widetilde{u_{(\varepsilon)}}\cdot\nabla)\,\Psi^\varepsilon,\nabla \Xi^\varepsilon\in
  L^{\frac{s}{s-n}}(0,T;L^{\frac{2s}{2n+s}}(\Omega))\qquad\text{uniformly in
  }\varepsilon>0.
\end{equation}
Observe that the norms of the above functions depend on the number $n\in \N$ of
applications of the maximal parabolic regularity, on the domain $\Omega$,  on the value of
$s>4$, on the norm $\|u\|_{r,s}$ (and the other data of the problem), but are independent
of the relevant quantities. 

We note that the case $s=4$ has been treated in~\cite{Gal2018}, hence we will not study it
(even if comes after a single simpler step). 

We distinguish between two cases depending on the fact that $s\in \R$ is or not an even
number. 

In the case $s>4$ is an even number, that is  $s=2m$, with $2<m\in\N$, it follows that with $m-1$
steps it holds
\begin{equation*}
\alpha_{m-1}=\frac{2s}{2+s}\qquad\text{and}\qquad   \beta_{m-1}=\frac{s}{s-1}, 
\end{equation*}
that is exactly the regularity from~\eqref{eq:final_estimate}.

In the case $s>4$ is not an even integer, we observe that $\beta_{n}$ is monotonically
strictly decreasing, hence we determine $N\in \N$ solving the inequalities
\begin{equation*}
  1<\beta_{{N+1}}< \frac{s}{s-1}\leq \beta_{N},
\end{equation*}
obtaining 
\begin{equation*}
 \frac{s}{2}-2\leq N<\frac{s}{2}-1
\end{equation*}
and the corresponding  value of $N\in \N$ is 
\begin{equation*}
  N:=\left[\frac{s}{2}\right]-1,
\end{equation*}
where $[x]$ denotes as usual the integer part of $x\in \R$.

At this point we can interpolate to obtain 
\begin{equation*}
  L^{\frac{s}{s-1}}(\Omega)=\left[L^{\beta_{N}}(\Omega),L^{\beta_{N+1}}(\Omega)\right]_{\theta}
  \qquad\text{with}\qquad \theta=  \frac{1}{2} \left(2 \left[\frac{s}{2}\right]-s+2\right) 
\end{equation*}
and consequently to obtain 
  with the same $\theta=  \frac{1}{2} \left(2 \left[\frac{s}{2}\right]-s+2\right)$
\begin{equation*}
\nabla \Xi,\ ((u-u_{(\varepsilon)}\cdot\nabla)\,\Psi^\varepsilon\in  \left[L^{\alpha_{N}}(0,T;L^{\beta_{N}}(\Omega)),L^{\alpha_{N+1}}(0,T;L^{\beta_{N+1}}(\Omega))\right]_{\theta}=
  L^{\frac{2s}{2+s}}(0,T;L^{\frac{s}{s-1}}(\Omega)),  
\end{equation*}
since, by direct computation
$\frac{2+s}{2s}=\frac{\theta}{\alpha_{N}}+\frac{1-\theta}{\alpha_{N+1}}$, ending the proof.
\end{proof}

\end{document}